\documentclass{amsart}

\numberwithin{equation}{section}
\usepackage{graphicx}
\usepackage{amssymb}
\usepackage{amsmath}
\usepackage{setspace}
\usepackage{color}
\allowdisplaybreaks

\usepackage{hyperref}
\usepackage{cleveref}
\usepackage[margin=3.6cm]{geometry}

\newcommand{\ra}{\rightarrow}

\newcommand{\IMS}{\mathrm{IMS}}
\newcommand{\Ind}{\mathrm{Ind}}
\newcommand{\tr}{\mathrm{tr}}
\newcommand{\Aut}{\mathrm{Aut}}

\newcommand{\mcal}{\mathcal}

\newcommand{\mbb}{\mathbb}

\newcommand{\Z}{\mathbb Z}
\newcommand{\R}{\mathbb R}
\newcommand{\C}{\mathbb C}

\newcommand{\1}{\mathbf{1}}
\newcommand{\mU}{\mathcal{U}}
\newcommand{\mA}{\mathcal{A}}
\newcommand{\mB}{\mathcal{B}}
\newcommand{\mC}{\mathcal{C}}
\newcommand{\mD}{\mathcal{D}}
\newcommand{\mE}{\mathcal{E}}
\newcommand{\mZ}{\mathcal{Z}}
\newcommand{\mM}{\mathcal{M}}
\newcommand{\mN}{\mathcal{N}}

\newcommand{\mQ}{\mathcal{Q}}

\newtheorem{theorem}{Theorem}[section]
\newtheorem{lemma}[theorem]{Lemma}
\newtheorem{cor}[theorem]{Corollary}
\newtheorem{definition}[theorem]{Definition}
\newtheorem{proposition}[theorem]{Proposition}
\newtheorem{remark}[theorem]{Remark}
\newtheorem{example}[theorem]{Example}

\begin{document}
\title[On various distances between subalgebras]{On various notions of distance
  between subalgebras of operator algebras}

\author[V P Gupta]{Ved Prakash Gupta} \author[S
  Kumar]{Sumit Kumar} \address{School of Physical Sciences, Jawaharlal Nehru University, New Delhi, INDIA}
\email{vedgupta@mail.jnu.ac.in} \email{sumitkumar.sk809@gmail.com}



\thanks{The second named author was  supported
  by the Council of Scientific and Industrial Research (Government
  of India) through a Senior Research Fellowship with
  CSIR File No. {\bf 09/0263(12012)/2021-EMR-I} }
\maketitle
\begin{abstract}
  Given  any irreducible inclusion $\mB \subset \mA$ of unital
$C^*$-algebras with a finite-index conditional expectation $E: \mA \to
\mB$, we show that the set of $E$-compatible intermediate
$C^*$-subalgebras is finite, thereby generalizing a finiteness result
of Ino and Watatani (from \cite{IW}). A finiteness result for a
certain collection of intermediate $C^*$-subalgebras of a
non-irreducible inclusion of simple unital $C^*$-algebras is also
obtained, which provides a $C^*$-version of a finiteness result of
Khoshkam and Mashood (from \cite{KM}).

Apart from these finiteness results, comparisons between
various notions of distance between subalgebras of operator algebras
by Kadison-Kastler, Christensen and Mashood-Taylor are made. Further,
these comparisons are used satisfactorily to provide some concrete
calculations of distance between  operator algebras associated to two
distinct subgroups of a given discrete group. \end{abstract}

\section{Introduction}
 Watatani (in \cite{Wat2}) (resp., Teruya and Watatani (in \cite{TW}))
 proved that the lattice of intermediate subfactors of an irreducible
 finite-index subfactor of type $II_1$ (resp., type $III$) is
 finite. This was then generalized to the $C^*$-context by Ino and
 Watatani (\cite{IW}), who proved that the set of intermediate
 $C^*$-subalgebras of an irreducible inclusion of simple unital
 $C^*$-algebras with a finite-index conditional expectation is finite
 (\cite[Corollary 3.9]{IW}). Further, Longo (in \cite{Longo}) obtained a
 bound for the cardinality of the lattice of intermediate subfactors
 of any finite-index irreducible inclusion of factors (of type $II_1$
 or $III$). More recently, a similar bound was obtained for the
 cardinality of the lattice of intermediate $C^*$-algebras of a
 finite-index irreducible inclusion of simple unital $C^*$-algebras by
 Bakshi and the first named author in \cite{BG}, which was achieved by
 introducing the notion of (interior) angle between intermediate
 $C^*$-subalgebras. This bound was further improved by Bakshi et
 al. in \cite{BE}.\smallskip

One of the highlights of this paper is the following generalization of
Ino-Watatani's finiteness result (from \cite{IW}):\smallskip

\noindent{\bf \Cref{IMS-finite}.} {\em Let $\mB \subset {\mA}$ be an
  irreducible inclusion of unital $C^*$-algebras with a finite-index
  conditional expectation $E: {\mA} \to \mB$. Then, the set $\IMS(\mB,
  {\mA}, E)$ consisting of $E$-compatible intermediate $C^*$-subalgebras
  of $\mB \subset \mA$ is finite.}
\smallskip

Ino-Watatani's proof of finiteness was a clever compactness argument
based on an appropriate estimate for $\|e_\mC - e_\mD\|$ for any two
($E$-compatible) intermediate $C^*$-subalgebras $\mC$ and $\mD$ (\cite[Lemma
  3.3]{IW}), and, a perturbation result established by them in
\cite[Theorem 3.8]{IW}. \Cref{IMS-finite}  is an  immediate consequence of the
following more general result, whose proof is an appropriate adaptation of
the compactness argument of Ino-Watatani (\cite{IW}):\smallskip

  \noindent{\bf \Cref{F1-F2-finite}.} {\it Let $\mB \subset {\mA} $ be an
    inclusion of unital $C^{*}$-algebras with a finite-index
    conditional expectation $E : \mA \rightarrow \mB$. If one
    (equivalently, any) of the algebras $\mcal{C}_{\mA}(\mB)$, $\mcal{Z}(\mB)$
    and $\mcal{Z}(\mA)$ is finite dimensional, then the collection
    \(
    \mathcal{F}(\mB , \mA, E):= \{\mC \in \IMS(\mB, \mA, E) : \mC_{\mA}(\mB)  \subseteq \mC_{\mA}(\mC) \cup \mC \}
    \)
    is  finite.
  }
\smallskip

It is noteworthy that in the above mentioned finiteness results
(except \Cref{F1-F2-finite}), irreducibility of the initial inclusion
was crucial. Somehow, not much is known about the finiteness of the
lattice of intermediate subalgebras of non-irreducible inclusions.
Interestingly, very recently, while looking for such finiteness
results for the full lattice of intermediate von Neumann subalgebras
of non-irreducible inclusions, the compactness argument of Ino and
Watatani was also employed successfully (in \cite{BG}) by Bakshi and
the first named author to prove the following:\smallskip

  \noindent{\bf Theorem.}  \cite[Theorem 6.4]{BG} {\it Let
    $\mathcal{N} \subset \mathcal{M}$ be an inclusion of von Neumann
    algebras with a faithful normal tracial state on $\mathcal{M}$
    such that $\mZ(\mathcal{N})$ is finite dimensional and the trace
    preserving normal conditional expectation from $\mathcal{M}$ onto
    $\mathcal{N}$ has finite Watatani index. If the relative commutant
    $\mathcal{N}'\cap \mathcal{M}$ equals either $\mZ(\mathcal{N})$ or
    $\mZ(\mathcal{M})$, then the lattice  consisting of intermediate von Neumann
    subalgebras of $\mathcal{N} \subset \mathcal{M}$ is
    finite.}\smallskip

 Prior to this,  Khoshkam and Mashood, in \cite[Theorem 1.3]{KM}, had shown
 that for any finite-index subfactor $N \subset M$ of type $II_1$, the
 subcollections \( \mathcal{L}_1(N \subset M):=\{P \in \mathcal{L}(N
 \subset M) : N'\cap M \subset P\} \) and \( \mathcal{L}_2(N \subset
 M):=\{P \in \mathcal{L}(N \subset M) : N'\cap M = P'\cap M\} \) are
 both finite, where $\mathcal{L}(N \subset M)$ denotes the lattice of
 intermediate subfactors of the inclusion $N \subset M$.

   Note that \Cref{F1-F2-finite} (as highlighted above) is a
   $C^*$-version of \cite[Theorem 1.3]{KM} for non-irreducible
   inclusions of non-simple $C^*$-algebras. Moreover, in
   \Cref{finiteness-results} itself, we also prove another variant of
   a $C^*$-version of \cite[Theorem 1.3]{KM} for non-irreducible
   inclusions of simple unital $C^*$-algebras, which reads as follows:\smallskip
  
  \noindent {\bf \Cref{L1-finite}.} {\it Let $\mB \subset \mA $ be an
    inclusion of simple unital $C^{*}$-algebras with a finite-index
    conditional expectation from $\mA $ onto $\mB$. Then, the
    sublattice $\mcal{I}_1(\mB \subset \mA):=\{ \mC \in \mcal{I}(\mB
    \subset \mA) : \mC_{\mA}(\mB) \subseteq \mC\}$ is finite.
  }\smallskip

A few words regarding the techniques employed to achieve the above mentioned finiteness results:

\Cref{F1-F2-finite} is achieved by directly employing the
  compactness argument of \cite{IW}, wherein the main ingredients
  include a basic observation from \cite{GS} (that $e_\mC \in \mA_1$ for
  any $\mC \in \IMS(\mB, \mA, E)$), an estimate for $\|e_\mC - e_\mD\|$ from
  \cite{IW} and a perturbation result from \cite{Dickson}. On the
  other hand, \Cref{L1-finite} is proved on the similar lines by
  first showing that $e_\mC \in \mA_1$ for any intermediate
  $C^*$-subalgebra $\mC \in \mcal{I}_1$ (\Cref{LCA-in-LBA}), then
  obtaining Ino-Watatani's estimate for $\|e_\mC - e_\mD\|$
  (\Cref{C-D-inequality}), where $\mC$ and $\mD$ are two intermediate
  $C^*$-subalgebras in $\mcal{I}_1$; and, finally, by exploiting a
  perturbation result by Dickson (from \cite{Dickson}).

  The above mentioned finiteness results are all achieved in
  \Cref{finiteness-results}. Prior to \Cref{finiteness-results}, we
  devote a short section (\Cref{prelims}) on preliminaries and another
  short section (\Cref{Kadison-Kastler}) recalling and proving some
  basic (yet interesting) observations related to the Kadison-Kastler
  distance. One interesting observation being that for any inclusion
  $\mB \subset \mA$ of unital $C^*$-algebras with a finite-index
  conditional expectation $E: \mA \to\mB$, there exists an $\alpha >
  0$ such that the set of unitaries
  \[
  \{u \in \mN_{\mA}(\mB) : \| u - \1 \| < \alpha \} \subseteq \bigcap \Big\{
  \mN_{\mA}(\mC): \mC \in  \mcal{F}(\mB , \mA, E) \Big\}
  \]
  - see \Cref{u-in-normalizer}. A slightly  stronger version
  holds for inclusions of simple unital $C^*$-algebras - see
  \Cref{u-in-normalizer-simple}. These two observations, respectively,
  are immediate consequences of the above mentioned perturbation
  results by Ino-Watatani and Dickson; and, a very simple minded
  inequality
  \[
  d_{KK}(\mB , u\mB u^*) \leq 2 d(u,
  \mN_{\mA}(\mB) \leq 2 \|u -1\|,
  \]
  obtained in \Cref{u-1-relation}, for
  every unitary $u$ in $\mA$.
    
    Then, in \Cref{comparisons}, we first recall and make some basic
    observations related to the different notions of distance between
    subalgebras of $C^*$-algebras and tracial von Neumann algebras,
    introduced by Christensen (in \cite{Ch1, Ch2, Ch3}) and Mashood
    and Taylor (in \cite{MT}). The essence of this section lies in
    making comparisons between them and the Kadison-Kastler distance;
    and, a desirable (and useful) observation made in:\smallskip

  \noindent {\bf \Cref{dMT-SOT}} {\it Let $\mM$ be a  von
    Neumann algebra with a faithful normal tracial state.  Then,
\[
d_{MT}\big(P, Q\big) = d_{MT}\Big(P, \overline{Q}^{S.O.T.}\Big) =
d_{MT}\Big(\overline{P}^{S.O.T.}, \overline{Q}^{S.O.T.}\Big),
\] 
for any two unital $*$-subalgebras $P$ and $Q$ of  $\mM$.    
  }
\smallskip

Finally, in \Cref{calculations}, we make use of the comparisons made
in \Cref{comparisons} and provide some concrete calculations of
various distances between subalgebras associated to distinct subgroups
(via $C^*$-crossed products, Banach group algebras and group von
Neumann algebras) of a given discrete group $G$. Interestingly, in all
cases, they turn out to be distance 1 apart - see
\Cref{distances-crossed-product-subalgebras},
\Cref{distance-Banach-gp-algebras} and \Cref{distances-gp-vNas}.

\section{Preliminaries}\label{index-basic-construction}\label{prelims}
We briefly recall the notions of finite-index conditional
expectations, Watatani's $C^*$-basic construction and compatible
intermediate $C^*$-subalgebras.

\subsection{Watatani's $C^*$-basic construction}
For any inclusion $\mB \subset \mA$ of unital $C^*$-algebras with common
unit and a faithful conditional expectation $E: \mA \to \mB$, $\mA$ becomes
a pre-Hilbert $\mB$-module with respect to the $\mB$-valued inner product
$\langle \cdot, \cdot \rangle_{\mB} : \mA \times \mA \to \mB$ given by $\langle
x, y \rangle_{\mB} = E(x^* y)$ for $x, y \in \mA$. We denote by $\mcal{E}$
the Hilbert $\mB$-module completion of $\mA$.

In order to distinguish the elements of the $C^*$-algebra $\mA$ and the
pre-Hilbert $\mB$-module $\mA$, following \cite{Wat}, we consider the
inclusion map $\eta : \mA \to \mA \subset \mcal{E}$. Thus,
\begin{equation}\label{eta-norm}
  \| \eta(x) \| := \|E(x^*x)\|^{1/2}\leq \|x\|
\end{equation}
for all $x \in \mA$. Let $\mathcal{L}_{\mB}(\mcal{E})$ denote the unital
$C^*$-algebra consisting of adjointable maps on $\mcal{E}$. Every
member of $\mathcal{L}_{\mB}(\mcal{E})$ is a $\mB$-module map (see \cite[$\S
  2.10$]{Wat}). Further, there is a natural $C^*$-embedding $\lambda: \mA
\to \mathcal{L}_{\mB}(\mcal{E})$ satisfying
\[
\lambda(a)\eta(x)=\eta(ax) \text{ for all } a,x \in \mA.
\]
Thus, we can identify $\mA$ with its image $\lambda(\mA)$ in
$\mcal{L}_{\mB}(\mcal{E})$. Further, there is a natural projection
$e_1 \in \lambda(\mB)'\cap \mathcal{L}_{\mB}(\mcal{E})$
(called the Jones projection corresponding to $E$) satisfying, via the
above identification, the relations
\[
e_1(\eta(x)) =\eta(E(x))\text{ and } e_1 x e_1 = E(x) e_1
\]
for all $x \in \mA$.  The Watatani's $C^*$-basic construction for the
inclusion $\mB \subset \mA$ with respect to the conditional
expectation $E$ is  defined as the $C^*$-algebra
\[
\mA_1=\overline{\mathrm{span}}\{ x e_1 y : x ,y \in \mA\} \subseteq
\mathcal{L}_{\mB}(\mcal{E}).
\]
Like the `Jones' basic construction for
a subfactor, Watatani's basic construction also satisfies a natural
universal property - see \cite[Proposition 2.2.11]{Wat}.  \smallskip

\noindent {\bf Some standard notations:} Recall that for an inclusion $\mB \subset \mA$ of $C^*$-algebras, the
centralizer of $\mB$ in $\mA$ is defined by
\[
\mcal{C}_{\mA}(\mB) = \{ x \in \mA: bx = xb \text{ for all } b\in \mB\};
\]
and, if $\mB \subset \mA$ are unital (with common unit), the normalizer of $\mB$ in $\mA$ is defined by
\[
\mcal{N}_{\mA}(\mB) = \{ u \in \mcal{U}(\mA): u \mB u^* = \mB\}.
\]
The centralizer $\mcal{C}_{\mA}(\mB)$ is a $C^*$-subalgebra of $\mA$
and is also denoted by $\mB'\cap \mA$ and called the relative
commutant of $\mB$ in $\mA$. The normalizer $\mN_{\mA}(\mB)$
is a closed subgroup of $\mcal{U}(\mA)$.

\subsection{Finite-index conditional expectations and index}
For any inclusion $\mB \subset \mA$ of unital $C^*$-algebras, a
conditional expectation $E: \mA \to \mB$ is said to have finite index
if there exists a finite set $\{\lambda_i: 1 \leq i \leq n\} \subset
\mA$ such that $x = \sum_{i=1}^n E(x \lambda_i) \lambda_i^*$ for all
$x \in \mA$. Such a set $\{\lambda_i\}$ is called a quasi-basis for
$E$ and the Watatani index of $E$ is defined as $\Ind_W(E) =
\sum_{i=1}^n \lambda_i \lambda_i^*$, which is a positive invertible
element in $\mZ(\mA)$ and is independent of the quasi-basis
$\{\lambda_i\}$.  Every such $E$ is faithful and preserves the unit, i.e., $1_{\mB} = E(1_{\mA}) = 
1_{\mA}$.

The set of finite-index conditional expectations from $\mA$ onto $\mB$
is denoted by $\mE_0(A, B)$. If $\mZ(\mA) = \C$ and $\mE_0(A, B)\neq
\emptyset$, then an $F \in \mE_0(\mA, \mB)$ is said to be minimal if
\[
\mathrm{Ind}_W(F) = \inf \{
\Ind_W(E): E \in \mE_0(\mA, \mB)\}.
\]
In some nice cases, a minimal conditional expectation
exists  and is unique as well.

\begin{remark}\label{index-defn}
Let $\mB\subset \mA$ be an inclusion of
unital $C^*$-algebras with $\mE_0(\mA, \mB) \neq \emptyset$ and 
$\mZ(\mA) = \C = \mZ(\mB)$. 
\begin{enumerate}
  \item There exists a unique minimal conditional expectation (denoted usually by) $E_0 : \mB \to \mA$. (\cite[Theorem 2.12.3]{Wat})
\item 
 The index of such an inclusion is defined as\hfill \hspace*{5mm}
 (\cite[Definition 2.12.4]{Wat})
  \begin{equation}
    [\mA : \mB]_0 = \Ind_W(E_0) = \min \{ \Ind_W(E): E \in \mE_0(\mA, \mB).
  \end{equation}
\end{enumerate}
\end{remark}

\subsubsection{Compatible intermediate $C^*$-subalgebras}

For any inclusion $\mB \subset \mA$ of unital $C^*$-algebras (with
common unit), let $\mathcal{I}(\mB \subset \mA)$ denote the collection
of intermediate $C^*$-subalgebras of the inclusion $\mB \subset \mA$
and let $\mcal{E}(\mA, \mB)$ denote the set of conditional
expectations from $\mA$ onto $\mB$. Further, for any $E \in
\mcal{E}(\mA, \mB)$, following \cite{IW} (also see \cite{BG, GS}), let
\[
\mathrm{IMS}(\mB, \mA, E):= \big\{\mC \in \mcal{I}(\mB \subset \mA):
\text{ there exists an } F\in \mcal{E}( \mA, C) \text{ such that }
E_{\restriction_C}\circ F = E\big\}.
\]
If $E$ has finite index, then for any $\mC \in \mathrm{IMS}(\mB, \mA,
E)$, a compatible conditional expectation from $\mA$ onto $\mC$ is
unique and has finite index (see \cite[Page 471]{IW}).

\subsubsection{Dual conditional expectation and iterated tower of basic constructions}

\begin{remark}\label{Dual-exp} \label{A1=span}
  Let $\mB\subset \mA$ be an inclusion of unital $C^*$-algebras and
  $E: \mA\rightarrow \mB$ be a finite-index conditional
  expectation. Let $\mA_1$ denote the $C^*$- basic construction of
  $\mB\subset \mA$ with respect to the conditional expectation $E$ and
  $e_1$ denote the corresponding Jones projection. The following facts
  are noteworthy and shall be needed ahead:
  \begin{enumerate}
\item $\sum_i \lambda_i e_1 \lambda_i^* = 1$ for any quasi-basis $\{\lambda_i\}$ of $E$.
  \item\(
\mA_1 = \mathrm{span}\{x e_1 y : x, y \in \mA\} =
C^*(\mA, e_1) = \mathcal{L}_{\mB}(\mA)\). \hfill (\cite[Proposition 1.3.3]{Wat})
  \item There exists a unique finite-index
  conditional expectation $\widetilde{E}: \mA_1 \rightarrow \mA$ satisfying
\[
\widetilde{E}(x e_1 y)= (\Ind_W(E))^{-1}xy
\]
for all $x, y \in \mA$.  \hspace{0.3cm}(\cite[Proposition
  1.6.1]{Wat})\\ ($\widetilde{E}$ is called the dual conditional
expectation of $E$ and is often  denoted by $E_1$.) 
 
\item By iteration, one obtains a tower of unital
  $C^*$-algebras  (see \cite[$\S 3.1$]{JOPT}, \cite[Proposition 3.18]{KAW} and
  \cite[$\S 2.3$]{BG}))
\[
\mA_{-1}:=\mB \subset \mA_0:=\mA \subset \mA_1 \subset \mA_2 \subset \cdots \subset
\mA_k \subset \cdots 
  \]
with finite-index conditional expectations $E_k : \mA_k \to
\mA_{k-1}$ and Jones projections $e_k \in \mA_k$, $k \geq
1$ with $\mA_{k} = C^*(\mA_{k-1} , e_k)$ for all $k \geq
1$. Also, $E_{k+1}$ is the dual of $E_k$ for all $k \geq 0$, with $E_0:= E$.

\item If $\Ind_W(E ) \in B$, then $ \Ind_W(\widetilde{E
})= \Ind_W(E ) $. Morevoer, if $\mA$ and $\mB$ are both simple and $E$ is minimal,
  then $\mA_1$ is simple and $\widetilde{E}$ is
  minimal. (\cite[Cor. 2.2.14, Prop. 2.3.4]{Wat} and \cite[Cor.
    3.4]{KAW})

\item (Pushdown Lemma.) For each $x_1\in \mA_1$, there exists a unique $x_0\in\mA$
 such that $x_1e_1= x_{0}e_1$ and $x_0$ is given by $x_0 = \Ind_{W}(E)
 \widetilde{E}(x_1 e_1)$. \hfill \hspace{0.3cm}(\cite[Lemma 3.7]{JOPT})
  \end{enumerate}
  \end{remark}

In \cite[Proposition 3.2]{BG}, it was shown (using the so-called
``Fourier transforms'') that if $\mB \subset \mA$ is an inclusion of
simple unital $C^*$-algebras with $\mE_0(\mA, \mB) \neq \emptyset$,
then $\mB'\cap \mA_{k} \cong \mA'\cap \mA_{k+1}$ (as vector spaces)
for all $k \geq 0$. In particular, if $\mB \subset \mA$ is irreducible
then so is $\mA \subset \mA_1$. It turns out that the last inference
is true for more general inclusions and will be needed ahead.

\begin{lemma}\label{irreducible-tower}
Let $\mB \subset \mA$ be an  inclusion of unital
$C^*$-algebras with a finite-index conditional expectation $E: \mA \to
\mB$. Then, 
\begin{enumerate}
\item $ E_{k}(\mA_{k-2}'\cap \mA_{k})= \mA_{k-2}' \cap \mA_{k-1}  $ for
  every $k \geq 1$; and,
  \item if, in addition, the inclusion $\mB \subset \mA$ is
    irreducible, then the inclusions $\mA_{k-1}\subset \mA_{k}$, $k
    \geq 1$ are all irreducible.
\end{enumerate}
\end{lemma}
\begin{proof} It is enough to prove for $k =1$. 

  (1):  Clearly, for  each $x_1\in \mB'\cap \mA_1$, 
  \(
  \widetilde{E}(x_1) b = \widetilde{E}(x_1b) =
  \widetilde{E}(bx_1) = b \widetilde{E}(x_1)
  \) for all $b \in \mB$.   Hence,
  $ \widetilde{E}(\mB'\cap \mA_1) \subseteq \mB'\cap \mA$.

  For the reverse inclusion, note that,  for each $a \in \mB'\cap \mA$,
  $x_1:=\Ind_W(E) ae_1 \in \mB'\cap \mA_1$ and
  $\widetilde{E}(x_1) = a$, by \Cref{Dual-exp}. Hence, $\mB'\cap \mA
  \subseteq \widetilde{E}(\mB'\cap \mA_1)$.\smallskip

  (2): Now, suppose that $\mB \subset \mA$ is irreducible and $x_1 \in
  \mA'\cap \mA_1$. Let $\{\lambda_i: 1 \leq i \leq n\} \subset \mA$ be
  a quasi basis for $E$. Then, $x_0 := \Ind_W(E) \widetilde{E}(x_1
  e_1)\in \mA$ and $x_1e_1= x_{0}e_1$, by \Cref{Dual-exp}(6). Further, for every $b\in B$, we have
\[
x_{0}b = \Ind_W(E)\widetilde{E}(x_1e_1)b= \Ind_W(E)\widetilde{E}(x_1e_1b)=
\Ind_W(E) \widetilde{E}(bx_1e_1)= b \Ind_W(E)\widetilde{E}(x_1e_1)= bx_0.
\]
Hence, $x_0\in \mB^{'}\cap \mA =\C$; so that, $x_0= \beta\1$ for some $\beta\in \mbb{C}$, which then shows that 
\[
x_1  =  x_1\sum_{i=1}^n \lambda_{i}e_1\lambda_{i}^* 
 =  \sum_{i=1}^n \lambda_{i}x_1e_1\lambda_{i}^*
 =   \beta\sum_{i=1}^n \lambda_{i} e_1\lambda_{i}^*
 =  \beta \1,
 \]
where the first equality holds because of \Cref{A1=span}(1).
 This implies that $\mA^{'}\cap \mA_1 =  \mbb{C}\1$, and we are done.
 \end{proof}

See \cite{Wat, IW} for more on Watatani index,
$C^*$-basic construction and compatible intermediate $C^*$-subalgebras.

\section{Kadison-Kastler distance}\label{Kadison-Kastler}
For any normed space $X$, as is standard, its closed unit ball will be
denoted by $B_1(X)$ and for any subset $K$ of $X$ and an element $x
\in X$, the distance between $x$ and $K$ is defined as
\[
d(x, K) = \inf\{ \|x - y\|: y \in K\}.
\]

\begin{definition}  \cite{KK}
  The Kadison-Kastler distance between any two subalgebras $\mC$ and $\mD$
  of a normed algebra $\mA$ (which we denote by $d_{KK}(\mC,\mD)$) is
  defined as the Hausdorff distance between the closed unit balls of
  $\mC$ and $\mD$, i.e.,
\[
d_{KK}(\mC,\mD) = \max\left\{\sup_{x \in B_1(\mC)} d(x, B_1(\mD)), \sup_{z\in B_1(\mD)}d(z, B_1(\mC))\right\}.
\]
\end{definition}

The Kadison-Kastler distance makes sense even for two subspaces of a
normed space but we shall work mainly with the distance between
subalgebras of a normed algebra.

We must remark that the notation $d_{KK}$ is not standard. We have
used it for the Kadison-Kastler distance in order to keep a
distinction between it and two other notions of distance introduced by
Christensen and a distance introduced by Mashood and
Taylor, which shall be discussed in \Cref{comparisons}.

\noindent {\bf Notations:}  For a normed algebra $\mA$,  let
\[
\mathrm{Sub}_{\mA} :=\big\{ \text{subalgebras of } \mA \big\}; 
\]
\[
C\text{-}\mathrm{Sub}_{\mA}:=\big\{ \text{closed subalgebras of } \mA\big\};
\]
and, if $\mA$ is a $C^*$-algebra, then let
\[
C^*\text{-}\mathrm{Sub}_{\mA} :=\big\{ C^*\text{-subalgebras of } \mA\big\}.
\]

Here are some well known elementary observations related to the
Kadison-Kastler distance.
\begin{remark}\label{basic-facts}
  Let $\mA$ be a normed algebra. 
\begin{enumerate}
\item $d_{KK}(\mC,\mD) \leq 1$ for all $\mC, \mD \in \mathrm{Sub}_{\mA}$.
\item If $\mA$ is a Banach algebra, then $d_{KK}$ is a metric on
  $C\text{-}\mathrm{Sub}_{\mA}$.

\item If $\mC, \mD \in C\text{-}\mathrm{Sub}_{\mA}$ and $\mC$ is a proper
  subalgebra of $\mD$, then $d_{KK}(\mC,\mD) =1$. (\cite[Lemma 2.1]{IW})
\end{enumerate}
\end{remark}

The following elementary observation is obvious and will be used to
calculate the distance between certain $C^*$-subalgebras in \Cref{calculations}.

\begin{lemma}\label{dense-distance}
Let $\mA$ be a normed algebra. Then,
\[
d_{KK}(\mC, \mD) = d_{KK}\big(\mC, \overline{\mD}\big) = d_{KK}\big(\overline{\mC}, \mD\big) =
d_{KK}\big(\overline{\mC},\overline{ \mD}\big)
\]
for all $\mC, \mD \in \mathrm{Sub}_{\mA}$.
\end{lemma}
It is natural to ask whether for a given subalgebra $\mB$ of a
$C^*$-algebra $\mA$, can one always find a subalgebra as close as one
desires. In \cite[Example, 2.2.2]{IN}, it was shown that $d_{KK}(\mB,
u \mB u^*) \leq \|u - \1\|$ for all $u \in \mU(\mA)$. Thus, because of
the following lemma, one can get a conjugate subalgebra as close as
one wishes.  (Its proof follows from an elementary continuous
functional calculus argument and we leave it to the reader.)

\begin{lemma}\label{close-unitary}
Let $\mathcal{A}$ be a unital $C^{*}$-algebra. Then, for each
$\epsilon > 0$, there exists a unitary $ u$ in  $\mathcal{A}$ such
that $0< \|u- \1\| < \epsilon$. 
\end{lemma}

Further, for any inclusion $\mB \subset \mA$ of unital
$C^*$-algebras, it is also quite natural to ask whether for a unitary
$u \in \mathcal{U}(A)$, there exists any relationship between its
distance from $ \mN_{\mA}(\mB)$ and the Kadison-Kastler distance
between $\mB$ and its conjugate $u \mB u^*$.

Interestingly, motivated by an inequality given by Popa-Sinclair-Smith
(\cite[Lemma 6.3]{Po}), we obtain the following pleasing relationship
without much effort, which then has a nice consequence that if $E \in
\mathcal{E}_0(\mA, \mB)$, then for any $\mC \in \mcal{F}(\mB, \mA, E)$
(see (\ref{FBAE-expression})), a unitary which normalizes $\mB$ and is
sufficiently close to $\mN_{\mA}(\mC)$ must belong to $\mN_{\mA}(\mC)$
- see \Cref{u-close-to-normalizer}.

\begin{lemma}\label{Normaliser-distance}\label{u-1-relation}
Let $\mA$ be a unital $C^*$-algebra and $\mB$ be a unital
$C^*$-subalgebra of $\mA$. Then, 
$$
d_{KK}\big(\mB, u\mB u^{*}\big)\leq 2d\big(u, \mN_{\mA}(\mB)\big) \leq 2 \|u - \1\|
$$
for all $u\in \mcal{U}(\mA)$.
\end{lemma}

\begin{proof}
  Let $u \in \mU(\mA)$. Clearly, $d\big(u, \mN_{\mA}(\mB)\big) \leq \|u - \1\|$.

  Next, let $v\in \mN_{\mA}(\mB)$ and $w:=
uv^{*}$. Then, $w \mB w^{*}= (uv^{*})\mB(uv^*)^*= u(v^{*}\mB v)u^{*} =
u\mB u^{*}$. So, for any $x\in B_1(\mB)$,
\[
\|x-wxw^*\|  =  \|xw- wx\| \leq  \|xw-x\| +\|x-wx\| 
\leq  2\|w- \1\|.
\]
This implies that $d\big(x, B_1(w \mB w^{*})\big)\leq 2\|w- \1\|.$ Hence,
$$
\sup_{x \in B_{1}(\mB)}d\Big(x, B_1(w\mB w^{*})\Big)\leq
2\|w- \1\|.
$$
Likewise,
\[
\sup_{y \in B_1(w\mB w^{*})}d\big(y, B_1(\mB)\big)\leq
2\|w- \1\|.
\]
So,
$$
d_{KK}\big(\mB,w\mB w^{*}\big) \leq 2 \|w-\1\|.
$$ Thus, $d_{KK}\big(\mB,w\mB w^{*}\big) \leq 2 \|u-v\|$ for all $v\in
N_{\mA}(\mB)$. Hence,\\ \hspace*{25mm} $ d_{KK}\big(\mB,u\mB u^{*}\big)= d_{KK}\big(\mB,w\mB
w^{*}\big) \leq 2 d(u, N_{\mA}(\mB)).  $
\end{proof}

In the reverse direction, Kadison and Kastler (in \cite{KK}) had
conjectured that sufficiently close subalgebras must be conjugates of
each other.  This was answered in the affirmative for various cases in
a series of some fundamental papers by Christensen, Phillips, Raeburn
and others in the decade of 70s.  Since then there have been
several other such so-called ``perturbation results''. People have
also employed such perturbation results to answer other important
questions. One such result with a nice application is due to Ino and
Watatani (from \cite{IW}). In fact, some of the results in this article are
direct applications of the perturbation result of Ino and Watatani.

Consider the following question, which arises naturally from
\Cref{u-1-relation}:\smallskip

\noindent {\em Question: Given a unital $C^*$-algebra $\mA$ and a
  $C^*$-subalgebra $\mB$, is every unitary sufficiently close to $\1$
  (or to $\mN_{\mA}(\mB)$) in the normalizer of the subalgebra
  $\mB$?}\smallskip

It is easily seen - see, for instance, \Cref{u-theta-S3} - that it has
a negative answer.  However, interestingly, \Cref{u-1-relation} along
with a perturbation result of Ino and Watatani (\cite[Proposition
  3.6]{IW}) immediately yields a somewhat positive answer for
compatible intermediate $C^*$-subalgebras - see
\Cref{u-close-to-normalizer} and \Cref{u-in-normalizer} below.

\noindent {\bf Notation:} Given an inclusion $\mB \subset \mA$ of
unital $C^*$-algebras with a finite-index conditional expectation $E:
\mA \to \mB$, let
\begin{equation}\label{FBAE-expression}
    \mathcal{F}(\mB,\mA, E) :=\big\{ \mC \in \IMS(\mB, \mA, E):
\mC_{\mA}(\mB) \subseteq \mC_{\mA}(\mC) \cup \mC\big\}.
\end{equation}
Note that $\mcal{F}(\mB, \mA, E) = \IMS(\mB, \mA, E)$ if
$\mC_{\mA}(\mB) \subseteq \mB$.

\begin{proposition}\label{u-close-to-normalizer}
Let $\mB \subset \mA$ be an inclusion of unital $C^*$-algebras with a
finite-index conditional expectation $E: \mA \to \mB$. Then, there
exists a constant $\alpha > 0$  such that
\(
\big\{ u \in {\mN}_{\mA}(\mB): d(u, \mN_{\mA}(\mC)) < \alpha \big\} \subseteq
 \mN_{\mA}(\mC)\) for every $\mC \in \mcal{F}(\mB, \mA, E)$.

In particular, if $\mC_{\mA}(\mB) \subseteq \mB$, then
\(
\big\{ u \in {\mN}_{\mA}(\mB): d(u, \mN_{\mA}(\mC)) < \alpha \big\} \subseteq
 \mN_{\mA}(\mC)\) for every $\mC \in \IMS(\mB, \mA, E)$.
\end{proposition}

\begin{proof}
Let $N$ be the number of elements (which is not unique) in a
quasi-basis for $E$ and let $\alpha = \frac{0.5}{(10 N)^4}$. First,
note that for any $\mC \in \IMS(B, A, E)$, with respect to the compatible
finite-index conditional expectation $F: \mA \to \mC$, and any  $u \in
\mU(\mA)$, there exists a faithful conditional expectation $F_u: \mA
\to u\mC u^*$ given by $F_u = \mathrm{Ad}_u \circ F \circ
\mathrm{Ad}_{u^{*}}$.

Now, let $\mC \in \mcal{F}(B, A, E)$ and $u \in \mN_{\mA}(\mB)$ with
$d( u, \mN_{\mA}(\mB)) < \alpha$. Then, $\mB \subseteq u \mC u^* \subseteq \mA $
and $d_{KK}(\mC, u \mC u^*) < \frac{1}{(10 N )^4}$, by
\Cref{u-1-relation}.  Thus, by \cite[Proposition 3.6]{IW}, there
exists a $v \in \mU(\mB'\cap \mA)$ such that $v \mC v^* = u \mC
u^*$. Since $\mB'\cap \mA \subseteq (\mC'\cap \mA) \cup \mC$, it
follows that $v \mC v^* = \mC$. Hence, $u \in \mN_{\mA}(\mC)$, and we
are done.
\end{proof}
\begin{cor}\label{u-in-normalizer}
Let $\mA, \mB$, $E$ and $\alpha$ be as in \Cref{u-close-to-normalizer}. Then,
\[
\big\{ u \in {\mN}_{\mA}(\mB): \|u - \1\| < \alpha\big\} \subseteq
\bigcap \Big\{ \mN_{\mA}(\mC): \mC \in \mcal{F}(\mB, \mA, E)\Big\}.
\]
In particular, if $\mC_{\mA}(\mB) \subseteq \mB$, then
\[
\big\{ u \in {\mN}_{\mA}(\mB): \|u - \1\| < \alpha\big\} \subseteq
\bigcap \Big\{ \mN_{\mA}(\mC): \mC \in \IMS(\mB, \mA, E)\Big\}.
\]
  \end{cor}
As an application of a perturbation result by Dickson (from \cite{Dickson})
and a fundamental result regarding finiteness of the index of a
conditional expectation by Izumi (from \cite{I}), we have a slightly more
general variant of \Cref{u-close-to-normalizer} for inclusions of
simple unital $C^*$-algebras. 

\noindent{\bf Notation:} For any inclusion $\mB \subset \mA$ of
unital $C^*$-algebras (with common unit), let
\begin{equation}\label{I0BAE-defn}
 \mcal{I}_0(\mB \subset \mA):=\big\{ \mC \in \mcal{I}(\mB \subset \mA):
\mC_{\mA}(\mB) \subseteq \mC_{\mA}(\mC) \cup \mC\big\}.
\end{equation}
Clearly, $ \mcal{F}(\mB, \mA, E) \subseteq \mcal{I}_0(\mB \subset
\mA)$ for every $E \in \mE_0(\mA, \mB)$; and, $
\mcal{I}_0(\mB \subset \mA) = \mcal{I}(\mB \subset \mA)$ if
$\mC_{\mA}(\mB) \subseteq \mB$.

\begin{proposition}\label{u-close-to-normalizer-simple}
Let $\mB \subset \mA$ be an inclusion of simple unital $C^*$-algebras
with a finite-index conditional expectation $E: \mA \to \mB$.  Then,
for any $0< \gamma < \frac{1}{10^{6}}$,
\(
\big\{ u \in {\mN}_{\mA}(\mB): d(u, \mN_{\mA}(\mC)) < \frac{\gamma}{2}\big\}
\subseteq  \mN_{\mA}(\mC)\) for every $\mC \in \mcal{I}_0(\mB \subset
\mA)$. 

In particular, if $\mC_{\mA}(\mB) \subseteq \mB$, then 
\(
\big\{ u \in {\mN}_{\mA}(\mB): d(u, \mN_{\mA}(\mC)) < \frac{\gamma}{2}\big\}
\subseteq  \mN_{\mA}(\mC)\) for every $\mC \in \mcal{I}(\mB \subset
\mA)$.
\end{proposition}

\begin{proof}
First, note that for any $\mC \in \mcal{I}(\mB \subset \mA)$, by
\cite[Proposition 6.1]{I}, there exists a conditional expectation $F:
\mA\rightarrow \mC$ of finite index. Hence, it is faithful.

Now, let $\mC \in \mcal{I}_{0}(\mB \subset \mA)$, $ 0 < \gamma <
\frac{1}{10^6}$ and $u \in \mN_{\mA}(\mB)$ with $d( u, \mN_{\mA}(\mC)) <
\frac{\gamma}{2}$. Then, $\mB \subseteq u \mC u^* \subseteq \mA $ and
$d_{KK}(\mC, u \mC u^*) < \gamma< 10^{-6}$, by
\Cref{u-1-relation}. Note that, as $E$ is of finite Watatani index, it
satisfies the Pimsner-Popa inequality; in particular, so does the
restriction $E_{\restriction u \mC u^*}: u \mC u^*\rightarrow \mB$.
Since $\mB$ is simple, it follows from \cite[Corollary 3.4]{I} that
$E_{\restriction u \mC u^*}$ also has a finite quasi-basis.  Thus, by
\cite[Theorem 3.7]{Dickson}, there exists a $v \in \mU(\mB'\cap \mA)$
such that $v \mC v^* = u \mC u^*$.  Since $\mB'\cap \mA \subseteq
(\mC'\cap \mA) \cup \mC$, it follows that $v \mC v^* = \mC$. Hence, $u
\in \mN_{\mA}(\mC)$, and we are done.
\end{proof}

\begin{cor}\label{u-in-normalizer-simple}
Let $\mA$, $\mB$ and $E$ be as in \Cref{u-close-to-normalizer-simple}.  Then,
for any $0< \gamma < \frac{1}{10^{6}}$,
\[
\Big\{ u \in {\mN}_{\mA}(\mB): \|u - \1\| < \frac{\gamma}{2}\Big\}
\subseteq \bigcap \Big\{ \mN_{\mA}(\mC): \mC \in \mcal{I}_0(\mB \subset
\mA)\Big\}.
\]
In particular, if $\mC_{\mA}(\mB) \subseteq \mB$, then 
\[
\Big\{ u \in {\mN}_{\mA}(\mB): \|u - \1\| < \frac{\gamma}{2}\Big\}
\subseteq \bigcap \Big\{ \mN_{\mA}(\mC): \mC \in \mcal{I}(\mB \subset
\mA)\Big\}.
\]
  \end{cor}

For any group $G$, let $\text{Sub}_G$ denote the collection of 
subgroups of $G$.  Then, $G$ admits a canonical action on
$\text{Sub}_G$ via conjugation, i.e.,
\[
G \times \text{Sub}_G \ni (g, H) \mapsto gH g^{-1} \in \text{Sub}_G.
\]
\begin{cor}
  Let $\mA, \mB, E$ and $\alpha $ be as in
  \Cref{u-close-to-normalizer} and let $G:=\mN_{\mA}(\mB)$. Then, with
  respect to the natural conjuation action of $G$ on $\mathrm{Sub}_G$,
  the following hold:
  \begin{enumerate}
\item     The open ball
  \[
  \{ u\in \mN_{\mA}(\mB) : \| u - \1\| < \alpha\}
  \subseteq \bigcap\big\{ \mathrm{Stab}_G\big(\mN_{\mA}(\mB) \cap
  \mN_{\mA}(\mC) \big) : \mC \in \mcal{F}(\mB, \mA, E)\big\};
  \]
and, in particular, if $\mC_{\mA}(\mB) \subseteq \mB$, then     
  \[
  \{ u\in \mN_{\mA}(\mB) : \| u - \1\| < \alpha\}
  \subseteq \bigcap\big\{ \mathrm{Stab}_G\big(\mN_{\mA}(\mB) \cap
  \mN_{\mA}(\mC) \big) : \mC \in \IMS(\mB, \mA, E)\big\}.
  \]
\item In addition, if $\mA$ and $\mB$ are both simple, then for any
  $0 < \gamma < \frac{1}{10^6}$, the open ball
 \[
 \Big\{ u\in \mN_{\mA}(\mB) : \| u - \1\| < \frac{\gamma}{2}\Big\}
  \subseteq \bigcap\big\{ \mathrm{Stab}_G\big(\mN_{\mA}(\mB) \cap
  \mN_{\mA}(\mC) \big) : \mC \in \mcal{I}_0(\mB \subset \mA)\big\};
  \]
and, in particular, if $\mC_{\mA}(\mB) \subseteq \mB$, then     
 \[
 \Big\{ u\in \mN_{\mA}(\mB) : \| u - \1\| < \frac{\gamma}{2}\Big\}
  \subseteq \bigcap\big\{ \mathrm{Stab}_G\big(\mN_{\mA}(\mB) \cap
  \mN_{\mA}(\mC) \big) : \mC \in \mcal{I}(\mB \subset \mA)\big\}.
  \]
  \end{enumerate}

\end{cor}

\section{Some finiteness results}\label{finiteness-results}

This section is devoted to some more applications of certain
perturbation results from \cite{IW} and \cite{Dickson} to generalize
some finiteness results by Ino-Watatani (\cite{IW}) and
Khoshkam-Mashood (\cite{KM}).

\subsection{Finiteness of certain compatible intermediate $C^*$-subalgebras }\( \)

\begin{theorem}\label{F1-F2-finite}
Let $\mB \subset \mA $ be an inclusion of unital $C^{*}$-algebras with
a finite-index conditional expectation $E : \mA \rightarrow \mB$. If
one (equivalently, any) of the algebras $\mcal{C}_{\mA}(\mB)$,
$\mcal{Z}(\mB)$ and $\mcal{Z}(\mA)$ is finite dimensional, then the
collection $\mathcal{F} (\mB,
 \mA, E)$ (as in (\ref{FBAE-expression})) is finite.
\end{theorem}
\begin{proof}
First,   note that,   since there exists a finite-index conditional
  expectation from $\mA$ onto $\mB$, it follows from \cite[Proposition
    2.7.3]{Wat} that the $C^*$-subalgebras $\mcal{C}_{\mA}(\mB)$,
  $\mcal{Z}(\mB)$ and  $\mcal{Z}(\mA)$ are either all finite
  dimensional or none of them is finite dimensional.

  Consider the Watatani's $C^*$-basic construction $\mA_1:=C^*(\mA,
  e_1)$ of the inclusion $\mB \subset \mA$ with respect to the
  conditional expectation $E$ and Jones projection $e_1$. In view of
  the preceding paragraph and the given hypothesis, $\mZ(\mB)$ is finite
  dimensional; so, by \cite[Proposition 2.7.3]{Wat} again, the
  relative commutant $\mB'\cap \mA_1$ is finite-dimensional.  Therefore,
  the set
\[
\mathcal{P}:=\{p\in \mB^{'}\cap \mA_1 :   p \mbox{\, is a projection }\}   
\]
is a compact Hausdorff space with respect to the operator norm.

Further, note that if $\mC \in \IMS(\mB, \mA, E)$ with respect to the
compatible conditional expectation $F: \mA \to \mC$, then $F$ has
finite index and $\mC_1 \subseteq \mA_1$ (see \cite[Prop. 2.7(2)]{GS}; so, the
corresponding Jones projection $e_{\mC}$ belongs to $\mC'\cap \mC_1
\subseteq \mB'\cap \mA_1$, where
$\mC_1$ denotes the $C^*$-basic construction of the inclusion $\mC
\subset \mA$ with respect to the finite-index conditional expectation
$F$ and Jones projection $e_{\mC}$.

Fix a $0 < \gamma < 10^{-6}$ and let $\varepsilon =
\frac{\gamma}{2\, \| \mathrm{Ind}_W(E)\|}$. By the compactness of
$\mcal{P}$, there exists a finite cover of $\mcal{P}$ consisting of
open balls of radius $\varepsilon$. So, it suffices to show that each
such $\varepsilon$-ball contains only finitely many Jones projections
corresponding to the members of $\mcal{F}(\mB, \mA, E)$.

Note that, for any two $\mC, \mD\in \IMS(\mB, \mA, E)$, their corresponding
Jones projections $e_{\mC}$ and $e_{\mD}$ are in $ \mB^{'}\cap \mA_1$ and
satisfy
$$
d_{KK}(\mC,\mD)\leq \|\mathrm{Ind}_W(E)\| \|e_{\mC}-e_{\mD}\|,
$$ by \cite[Lemma 3.3]{IW}. In particular, if $e_{\mC}$ and $e_{\mD}$ are
in one of the $\varepsilon$-open balls, then $d_{KK}(\mC, \mD)\leq
\frac{\gamma}{2}< \gamma$. Thus, by \cite[Theorem 3.7]{Dickson}, there exists a unitary $u$ in $\mB'\cap \mA$
such that $\mC = u\mD u^{*}$ (and $\|u- 1\|\leq
16\sqrt{110}\gamma^{\frac{1}{2}}+ 880 \gamma$).

Let $\mC, \mD \in \mathcal{F}(\mB, \mA, E)$ and they be in one
$\varepsilon$-ball as above.  As $u\in \mB^{'}\cap \mA$ and
$\mB^{'}\cap \mA \subseteq \mC_{\mA}(\mD) \cup \mD$, it follows that
$\mC=u\mD u^{*}= \mD$. Thus, each $\varepsilon$-open ball of the cover
contains at most one Jones projection for some member of
$\mcal{F}(\mB, \mA, E)$, as was desired.
\end{proof}
We can then immediately deduce the following generalization of
\cite[Corollary 3.9]{IW}. (The second part follows from
\Cref{irreducible-tower}.)

\begin{cor}\label{IMS-finite}
Let $\mB \subset \mA$ be an irreducible inclusion of unital
$C^*$-algebras with a finite-index conditional expectation $E: \mA \to
\mB$. Then, $ \IMS(\mB, \mA, E)$ is finite.

In particular, $
\IMS(\mA_{k}, \mA_{k+1}, E_{k+1})$ is finite for every $k \geq 0$.
  \end{cor}

For any unital inclusion of von Neumann algebras $\mcal{N}\subset
\mcal{M}$, let $\mcal{L}(\mcal{N}\subset \mcal{M})$ denote the lattice
of intermediate von Neumann subalgebras.  For any such inclusion, if
there exists a faithful normal tracial state $\tr$ on $\mcal{M}$ and
the unique $\tr$-preserving conditional expectation $E_{\mcal{N}}:
\mcal{M} \to \mcal{N}$ has finite Watatani index, then it was shown in
\cite{BG} that if $\mZ(\mN)$ is finite-dimensional and the relative
commutant $\mN'\cap \mM$ equals either $ \mZ(\mN)$ or $\mZ(\mM)$, then
$\mcal{L}(\mN \subset \mM)$ is finite. Analogous to \cite[Theorem
  1.3]{KM} and \Cref{F1-F2-finite}, we now have the following:

\begin{proposition}\label{vNa-finiteness-general}
Let $\mcal{N}\subset \mcal{M}$ be a unital inclusion of finite von
Neumann algebras with a (fixed) faithful normal tracial state $\tr$ on
$\mcal{M}$ such that the unique $\tr$-preserving conditional
expectation $E_{\mcal{N}}: \mcal{M} \to \mcal{N}$ has finite Watatani
index. If one (equivalently, any) of the algebras
$\mcal{C}_\mcal{M}(\mcal{N})$, $\mcal{Z}(\mcal{N})$ and
$\mcal{Z}(\mcal{M})$ is finite dimensional, then the subcollection
\[
\mathcal{L}_0(\mN \subset \mM): = \{\mcal{P} \in
\mcal{L}(\mcal{N}\subset \mcal{M}) : \mcal{N}{'}\cap \mcal{M}
\subseteq \mcal{P}\cup (\mcal{P}{'}\cap \mcal{M}) \}
\]
is finite.
\end{proposition}
\begin{proof}
Clearly, $\mcal{L}(\mcal{N}\subset \mcal{M}) \subseteq
\IMS(\mcal{N},\mcal{M},E_{\mcal{N}})$. Rest follows from
\Cref{F1-F2-finite}.
\end{proof}

Note that, \Cref{vNa-finiteness-general} also generalizes Watatani's
finiteness result \cite[Theorem 2.2]{Wat2}  to a non-irreducible setting.

\subsection{Finiteness results for non-irreducible inclusions of simple $C^*$-algebras}

In general, if $\mB \subset \mA$ are simple unital $C^*$-algebras and
the inclusion is not irreducible, then the lattice $\mathcal{I}(\mB
\subset \mA)$ need not be finite. For instance, consider the following
easy example:
\begin{example}
  Let $\mA=M_2(\mbb{C})$, $\mB=\mbb{C}I_2$, $\Delta=
  \{\mathrm{diag}(\lambda, \mu)\,\,:\,\,\lambda,\,\mu\,\in \mbb{C}\}$,
  $E: \mA\rightarrow \mB$ denote the canonical (tracial) conditional
  expectation given by
  \[
  E([a_{ij}])= \frac{(a_{11}+a_{22})}{2}I_2,\
  \ [a_{ij}]\in \mA;
  \]
  and, let $F: \mA \rightarrow
  \Delta$ denote another canonical conditional expectation given by
  $F([a_{ij}])= \mathrm{diag}(a_{11},a_{22})$, $[a_{ij}]\in \mA$.

  Note that $u\Delta
  u^*\in \IMS(\mB, \mA, E)$ for all $u \in U(2)$ (by \cite[Lemma
    2.8]{GS}). Also, the set $\{ u\Delta u^*: u\in U(2)\setminus
  \mcal{N}_{\mA}(\Delta)\}$ is infinite because the set of left cosets
  of $ \mcal{N}_{\mA}(\Delta)$ in $U(2)$ is infinite. Hence,
  $\IMS(\mB, \mA, E)$ and, therefore, $\mathcal{I}(\mB \subset \mA)$
  are infinite sets.

  Here is an indirect way of seeing why $U(2)/\mN_{\mA}(\Delta)$ is
  infinite. Suppose, on contrary, that $ \{[u]:=
  u\mcal{N}_{\mA}(\Delta) :  u\in U(2)\} $ is finite. Then, for
  any element $w\in u\mcal{N}_{\mA}(\Delta)$, $w=uv$ for some $v\in
  \mcal{N}_{\mA}(\Delta)$; so,
$$ \alpha(\Delta, w\Delta w^*)= \alpha(\Delta, uv\Delta
(uv)^*)=\alpha(\Delta, u\Delta u^*),
$$
where $\alpha$ is the interior angle (see \cite{GS}). This implies that \( \alpha(\Delta, w\Delta
w^*)=\alpha(\Delta, u\Delta u^*)  \)  for every $w\in [u]$. Hence, the set
\[\{ \alpha( \Delta, u \Delta u^*) :u \in U(2)\} = \left\{\alpha(\Delta,u\Delta
u^*)\,\,:\,\, [u]\in \frac{U(2)}{\mcal{N}_{\mA}(\Delta)}\right\}
\]
is finite. This contradicts the fact that $\{ \alpha( \Delta, u \Delta
u^*) :u \in U(2)\} = [0, \frac{\pi}{2}]$ (see \cite[Corollary
  4.6]{GS}). Thus, the set of left cosets of $\mN_{\mA}(\Delta)$ in
$U(2)$ must be infinite.
\end{example}

However, for a non-irreducible inclusion $\mB \subset
A$ of simple unital $C^*$-algebras with a finite-index conditional
expectation, we shall show in this section that the sublattice
consisting of intermediate $C^*$-subalgebras of $\mB \subset \mA$ which
contain  the centralizer algebra $\mC_{\mA}(\mB)$ is finite.

The following useful observation (whose first part comes from \cite{I}
and the second part is comparable with \cite[Proposition 2.7 (1,
  2)]{GS}) will be needed ahead.

\begin{proposition}\label{LCA-in-LBA}
Let $\mB \subset \mA$ be an inclusion of simple unital
$C^{*}$-algebras with a finite-index conditional expectation $E: \mA
\rightarrow \mB$. Then, for any $\mC\in \mcal{I}(\mB \subset \mA)$,
there exists a finite-index conditional expectation $F: \mA \to \mC$
and, moreover, if $\mC_{\mA}(\mB)\subseteq \mC $, then
$\mcal{L}_{\mC}(\mA)\subseteq \mcal{L}_{\mB} (\mA),$ where
$\mcal{L}_{\mC}(\mA)$ is defined with respect to $F$.

In particular, $\mC_1 \subset \mA_1$, where $\mA_1$ (resp., $\mC_1$) denotes
the $C^*$-basic construction of the inclusion $ \mB \subset \mA$ (resp.,
$\mC \subset \mA$) with respect to $E$ (resp., $F$).
\end{proposition}

\begin{proof}
Let $\mC\in \mcal{I}(\mB \subset \mA)$. Then, by \cite[Proposition 6.1]{I},
there exists a conditional expectation $F: \mA\rightarrow \mC$ of finite
index.

Next, suppose that $\mC_{\mA}(\mB)=\mB'\cap \mA \subseteq \mC$. 

Note that, as $E$ is of finite Watatani index, it satisfies the
Pimsner-Popa inequality (\cite{Wat}); in particular, so does the restriction
$E_{\restriction C}: \mC\rightarrow \mB$.  Since $\mB$ is simple, it then follows
from \cite[Corollary 3.4]{I} that $E_{\restriction_{\mC}}$ also has a
finite quasi-basis. Thus, $G:= E_{\restriction \mC}\circ F$ is of finite
Watatani index by \cite[Proposition 1.7.1]{Wat}. Since $E$ and $G$ are
two finite-index conditional expectations from $\mA$ onto $\mB$, by
\cite[Proposition 2.10.9]{Wat}, there exists a $q\in \mB^{'}\cap \mA $
such that
$
E(x)=G(q^*xq) \  \mbox{for all } x\in \mA.
$ 

Then, for any $T\in \mcal{L}_C(\mA)$, we observe that
\begin{eqnarray*}
\langle T(x),y\rangle_{\mB} &=& E(T(x)^{*}y)\\
& = & G(q^{*}T(x)^{*}yq)\\
&=& E_{\restriction_C}(F((T(x)q)^{*}yq)) \\
& = & E_{\restriction_C}(\langle T(x)q, yq\rangle_{C}) \\
& = & E_{\restriction C}(\langle xq,T^{*}(yq)\rangle_{C}) \\
&=& E_{\restriction C}\big(F(q^*x^*T^{*}(yq))\big)\\
& = & E_{\restriction C}\big(F(q^*x^*T^{*}(y)q)\big)\\
&=& G(q^*x^*T^{*}(y)q) \\
& = & E(x^*T^{*}(y))\\
& = & \langle x,T^*(y)\rangle_{\mB}
\end{eqnarray*}
for all $x,y\in \mA$. Hence, $T\in \mcal{L}_{\mB}(\mA)$.
\end{proof}

\noindent {\bf Notation:} For any inclusion $\mB \subset \mA$ of unital
$C^{*}$-algebras, let
$$
\mcal{I}_1(\mB \subset \mA): = \{\mC \in \mathcal{I}(\mB \subset \mA) : \mC_{\mA}(\mB)  \subseteq \mC \}.
$$ Clearly, $\mcal{I}_1$ a sublattice of $\mcal{I}(\mB \subset \mA)$,
$\mcal{I}_1(\mB \subset \mA) \subseteq \mcal{I}_0(\mB
\subset \mA)$ and $\mcal{I}_1(\mB \subset \mA) = \mcal{I}(\mB \subset
\mA)$ if $\mC_{\mA}(\mB) \subseteq \mB$.

We shall need the following adaptation of \cite[Lemma 3.3]{IW}.
\begin{lemma}\label{C-D-inequality}
Let $\mB \subset \mA$ be an inclusion of simple unital $C^{*}$-algebras with a
finite-index conditional expectation $E: \mA \rightarrow \mB$. Then, 
\begin{equation}\label{C-D-inequality-eqn}
d_{KK}(\mC,\mD)\leq \mathrm{Ind}_W(E)\, \|e_{\mC}-e_{\mD}\|
\end{equation}
for all $\mC,\mD \in \mathcal{I}_{1}(\mB \subset \mA)$.
\end{lemma}

\begin{proof}
  We shall simply write $\mcal{I}_1$ for the set $ \mcal{I}_{1}(\mB
  \subset \mA)$.

  Since $E$ has finite index, it satisfies  the Pimsner-Popa
  inequality   (by \cite[Proposition 2.6.2]{Wat}), i.e.,
  \(
  E(x^{*}x)\geq (\mathrm{Ind}_{W}(E))^{-1}x^{*}x
\)
  for all $x \in \mA$; and, since 
  $\Ind_W(E) \geq 1$ (by \cite[Lemma 2.3.1]{Wat}), it follows that
  \[
  E(x^{*}x)
  \geq (\mathrm{Ind}_{W}(E))^{-2}x^{*}x
  \]
  for all $x \in \mA$.  In particular, $\|\eta(x)\|\geq
  (\mathrm{Ind}_{W}(E))^{-1}\|x\|$ for all $x \in \mA$. Moreover,  $\|\eta(a)\|\leq \|a\|\leq 1$ for every
  $a\in B_{1}(\mC)$.

  Since $\mB \subset \mA$ are both simple,  it follows from \Cref{LCA-in-LBA} that
  $\mcal{L}_{\mC}(\mA) \cup \mcal{L}_{\mD}(\mA)\subseteq \mcal{L}_{\mB}
  (\mA)$, for any two $\mC, \mD \in
  \mcal{I}_1$. Thus, $e_{\mC}, e_{\mD} \in \mcal{L}_{\mB} (\mA)$, and
 $$ \|e_{\mC}-e_{\mD}\|\geq
\frac{\|\eta(E_{\mC}(a)-E_{\mD}(a))\|}{\|\eta(a)\|}\geq
\|\eta(a-E_{\mD}(a))\|\geq
\frac{1}{\mathrm{Ind}_{W}(E) }\|a-E_{\mD}(a)\|
$$ for all $a \in B_1(\mC)\setminus\{0\}$.  Therefore, for each $a\in
B_{1}(\mC)$, there exists an element $b:=E_{\mD}(a)\in B_{1}(\mD)$ such that
$\|a-b\|\leq \mathrm{Ind}_{W}(E)\, \|e_{\mC}-e_{\mD}\|$. By a symmetric
arguement, we deduce that
\[
d_{KK}(\mC,\mD)\leq
\mathrm{Ind}_W(E)\, \|e_{\mC}-e_{\mD}\|.
\]
\end{proof}

The following finiteness result is an adaptation (as well as a mild
generalization) of \cite[Corollary 3.9]{IW} (also see \cite[Theorem
  1.3]{KM}). Its proof is an imitation  of that of
\Cref{F1-F2-finite}. However, we provide all steps for the sake of
completeness.

\begin{theorem}\label{L1-finite}
Let $\mB \subset \mA $ be an inclusion of simple unital $C^{*}$-algebras
with a finite-index conditional expectation $E : \mA \rightarrow
\mB$. Then, the sublattice $\mcal{I}_1(\mB \subset \mA)$  is finite.
\end{theorem}

\begin{proof}
 Consider the Watatani's $C^*$-basic construction $\mA_1:= C^*(\mA,
  e_1)$ of the inclusion $\mB \subset \mA$ with respect to the
 finite-index conditional expectation $E$ and Jones projection
 $e_1$. Let $ \widetilde{E} : \mA_1 \rightarrow \mA$ be the
 finite-index dual conditional expectation of $E$. Then,
 $\widetilde{E} \circ {E} : \mA_1 \rightarrow \mB $ is also a
 conditional expectation of finite index. Since $\mA$ is simple, it
 follows from \cite[Proposition 2.7.3]{Wat} that the relative
 commutant $\mB'\cap \mA_1$ is finite-dimensional.  Therefore, the set
\[
\mathcal{P}:=\{p\in \mB^{'}\cap \mA_1 :   p \mbox{\, is a projection}\}   
\]
is a compact Hausdorff space with respect to the operator norm.

Since $\mA$ is simple, $\Ind_W(E)$ is a positive scalar. Fix a $0 <
\gamma < 10^{-6}$ and let $\varepsilon = \frac{\gamma}{2\,
  \mathrm{Ind}_W(E)}$. By the compactness of $\mcal{P}$, there
exists a finite cover of $\mcal{P}$ consisting of open balls of radius
$\epsilon$.  Now, it suffices to show that each such $\epsilon$-ball
contains only finitely many Jones projections corresponding to members
of $\mathcal{I}_{1}$.

Note that, for any two $\mC,\mD \in \mathcal{I}_{1}$, their
corresponding Jones projections $e_{\mC}$ and $e_{\mD}$ are in $
\mB^{'}\cap \mcal{L}_{\mB}(\mA)= \mB^{'}\cap \mA_{1}$, by
\Cref{LCA-in-LBA}; and, by \Cref{C-D-inequality}, they also satisfy the
inequality
$$
d_{KK}(\mC,\mD)\leq \mathrm{Ind}_W(E) \|e_{\mC}-e_{\mD}\|.
$$ 
So, if $e_{\mC}$ and $e_{\mD}$ are in one of the
$\varepsilon$-open balls, then $d_{KK}(\mC,\mD)\leq \frac{\gamma}{2}<
\gamma$. Thus, by \cite[Theorem 3.7]{Dickson}, there exists a unitary
$u$ in $\mB'\cap \mA$ such that $\mD= u\mC u^{*}$ (and $\|u-
1_{\mA}\|\leq 16\sqrt{110}\gamma^{\frac{1}{2}}+ 880 \gamma$). Since
$\mB'\cap \mA = \mC_{\mA}(\mB) \subseteq \mC$, this implies that $u\in
\mC$. Hence, $\mD=u\mC u^{*}= \mC$. This shows that each
$\varepsilon$-open ball of the cover contains at most one Jones
projection for some intermediate $C^{*}$-subalgebra in
$\mathcal{I}_{1}$, and we are done.
\end{proof}

Since countably decomposable type $III$ factors are simple, we obtain
the following generalization of \cite[Theorem 2.5]{TW} (and a partial
generalization of \cite[Theorem 1.3]{KM}).
\begin{cor}
Let $\mcal{N}\subset \mcal{M}$ be an inclusion of countably
decomposable type $III$ factors with a finite-index conditional
expectation. Then, the set $\mcal{L}_1(\mN \subset \mM):=\{
\mathcal{Q} \in \mcal{L}(\mN \subset \mM): \mN' \cap \mcal{M}
\subseteq \mcal{Q} \}$ is finite.
\end{cor}

\subsection{Cardinality of $\IMS(\mB, \mA, E)$}

Longo (in \cite{Longo}) showed that the cardinality of the lattice of
intermediate subfactors of any finite-index irreducible inclusion $N
\subset M$ of factors (of type $II_1$ or $III$) is bounded by
$([M:N]^2)^{[M:N]^2}$. More recently, a similar bound was obtained for
the lattice of intermediate $C^*$-algebras of a finite-index inclusion
of simple unital $C^*$-algebras by Bakshi and the first named author
in \cite{BG}, which was achieved by introducing the notion of (interior)
angle between two intermediate $C^*$-subalgebras. Further, Bakshi et
al (in \cite{BE}) improved the bound by proving that the cardinality
of the lattice of intermediate $C^*$-subalgebras of an irreducible
inclusion $\mB\subset \mA$ of simple unital $C^*$-algebras with finite
index is bounded by $9^{[\mA : \mB]_{0}}$.

In this subsection, for any irreducible inclusion $\mB \subset\mA$ of
unital $C^*$-algebras with a finite index conditional expectation $E: \mA \to \mB$,
through some elementary group action technique, we provide an
expression for the cardinality of $\IMS(\mB, \mA, E)$ in terms of the
indices of the unitary normalizers of its members in $\mN_{\mA}(\mB)$ - see
\Cref{IMS-cardinality}. On face, the expression looks a bit dry,
but it might prove useful while performing some concrete calculations.

In \cite[Lemma 2.8(2)]{GS}, it was shown that for any inclusion $\mB
\subset \mA$ of unital $C^*$-algebras with a tracial finite-index
conditional expectation $E: \mA \to B$, $u \mC u^* \in \IMS(\mB, \mA, E)$
for every $\mC\in \IMS(\mB, \mA, E)$ and $u \in \mN_{\mA}(\mB)$. We now
have its following (more obvious) variant (which does not require $E$
to be tracial), which gives an abundance of compatible intermediate
$C^*$-algebras and will also allow us to get an expression for the
cardinality of $\IMS(\mB, \mA, E)$ when the inclusion $\mB \subset
\mA$ is irreducible.

\begin{lemma}\label{IMS-elements}
  Let $\mB\subset \mA$ be an inclusion of unital $C^*$-algebras with a 
  finite-index conditional expectation $E: \mA\rightarrow \mB$ such that
  $\mC_{\mA}(\mB) \subseteq \mB$.  Let  $\mC\in \IMS(\mB, \mA, E)$
  with respect to the compatible finite-index conditional expectation
  $F: \mA\rightarrow \mC$. Then, for every $u\in \mN_{\mA}(\mB)$, $u\mC
  u^*\in \IMS(\mB, \mA, E)$ with respect to the conditional expectation
  $F_u: \mA\rightarrow u\mC u^*$ given by $F_u= \mathrm{Ad}_u\circ
  F\circ \mathrm{Ad}_{u^*}$.
\end{lemma}
\begin{proof}
As $\mC_{\mA}(\mB)  \subseteq \mB$ and $\mathcal{E}_0(\mA , \mB) \neq
\emptyset$, it follows from \cite[Corollary 1.4.3]{Wat} that $E$ is
the unique  conditional from $ \mA$ onto $ \mB$. Also,
from \cite[Lemma 2.8(1)]{GS} we know that $F_u$ is of finite
index. So, $E_{\restriction_{u\mC u^*}}\circ F_u$ is again a finite-index
conditional expectation from $\mA$ onto $\mB$. Thus, by uniqueness of
$E$, we get $E_{\restriction_{u\mC u^*}}\circ F_u=E$. Hence, $u\mC u^*\in
\IMS(\mB, \mA, E)$ with respect to the compatible conditional
expectation $F_u$.
\end{proof}
  With notations as in \Cref{IMS-elements}, we thus observe that the group
  $G:=\mN_{\mA}(\mB)$ admits an action on the set $\IMS(\mB, \mA, E)$ via
  the map
  \[
 G \times \IMS(\mB, \mA, E) \ni (u, \mC) \mapsto u\mC
  u^*\in \IMS(\mB, \mA, E).
  \]
For any $\mC\in \IMS(\mB, \mA, E)$, its stabilizer
\[
\mathrm{Stab}_G(\mC)= \mN_{\mA}(\mB)\cap \mN_{\mA}(\mC).
\]
  Let $\widehat{\IMS(\mB, \mA, E)}$ denote a set of representatives of
  the orbits of $\IMS(\mB, \mA, E)$ with respect to this action. Then, in
  view of \Cref{IMS-finite}, from some basic theory of group actions,
  we immediately conclude the following:

\begin{proposition}\label{IMS-cardinality}
Let $\mB\subset \mA$ be an irreducible inclusion of unital
$C^*$-algebras with a finite-index conditional expectation $E:
\mA\rightarrow \mB$ and $G:=\mN_{\mA}(\mB)$. Then,
\begin{enumerate}
\item $\big[\mN_{\mA}(\mB) : \mN_{\mA}(\mB)\cap \mN_{\mA}(\mC)\big] < \infty$
  for every $C\in \IMS(\mB, \mA, E)$; and,
\item \( \big|\IMS(\mB, \mA, E)\big|= \sum_{\mC\in \widehat{\IMS(\mB,
    \mA, E)}} \big[\mN_{\mA}(\mB) : \mN_{\mA}(\mB)\cap \mN_{\mA}(\mC)\big].
  \)
\end{enumerate}
\end{proposition}

\begin{lemma}
Let $\mA, \mB$, $E$ and $G$ be as in \Cref{IMS-cardinality}
and $\mC\in \IMS(\mB, \mA, E)$. Then,
\[
[\mA:\mD]_{0}= [\mA : \mC]_{0}\ \text{and} \
[\mD:\mB]_{0}= [\mC : \mB]_{0}
\]
for very $\mD \in \mathrm{Orb}_G(\mC)$.
 \end{lemma}

\begin{proof}
  Since $\mB\subset \mA$ is an irreducible inclusion, the conditional
  expectation $E:\mA  \rightarrow \mB$ is unique, by \cite[Corollary
    1.4.3]{Wat}. Hence, it is minimal.

  Suppose that $\mC\in \IMS(\mB, \mA, E)$ with respect to the
  compatible finite-index conditional expectation $F: \mA\rightarrow
  C$. Note that $\mathrm{Orb}_G(\mC) =
  \{ u \mC u^* : u \in G\}$.

Let $\mD \in \mathrm{Orb}_G(\mC)$. Then, $\mD = u \mC u^*$ for some $u
\in \mN_{\mA}(\mB)$ and $\mD \in \IMS(\mB, \mA, E)$ with respect to
the conditional expectation $F_u:=\mathrm{Ad}_u \circ F \circ
\mathrm{Ad}_{u^*}$. Also, $\mZ(\mD)= \mZ(u \mC u^*) = u\mZ(\mC)u^* = \C =
\mZ(\mA)$, $\mD{'}\cap \mA \subseteq \mB{'}\cap \mA =\C$ and
$\mB{'}\cap \mD \subseteq \mB{'}\cap \mA=\C$; so, by \cite[Corollary
  1.4.3]{Wat} again, $F_u$ and $E_{\restriction_{\mD}}$ are both
unique and, hence, minimal. Thus, from \cite[Theorem 3]{KW}, we get \(
[\mA:\mB]_0=[\mA: \mD]_0[\mD:\mB]_0 \) for every $\mD \in
\mathrm{Orb}_G(\mC)$. Note that, $[\mA: \mD]_{0} = [\mA: u\mC u^*]_{0} =
       [\mA: \mC]_0$, by \cite[Lemma 2.8]{GS}, which then shows that
       \( [\mD :\mB]_{0}= [\mC: \mB]_{0}\), as well.
\end{proof}

\section{Comparisons between various notions of distance}
\label{comparisons}

In this section, we discuss the various notions of distance between
subalgebras of a given operator algebra by Christensen and
Mashood-Taylor and make comparisons between them and the
Kadison-Kastler distance.

\subsection{Christensen's  distances}

Christensen has given two notions of distances and used them
effectively in proving some significant perturbation results.  The
first one is defined between subalgebras of any tracial von Neumann
algebra (\cite{Ch1}) and, more generally, the other one is
defined between subalgebras of any $C^*$-algebra (\cite{Ch2}). We
briefly recall both notions in reversed order and state some facts
related to them.

 \subsubsection{Christensen's distance between subalgebras of a normed algebra}

 Though Christensen gave the notion of distance between subalgebras of a
 $C^*$-algebra, the same can be used in the normed algebra context as
 well.

 Let $A$ be a normed algebra. Recall from \cite{Ch2} that, for $C,\, D
 \in \mathrm{Sub}_A$ and a scalar $\gamma > 0$, $C \subseteq_{\gamma}
 D$ if for each $x \in B_1(C)$, there exists a $
 y \in D$ such that $\|x-y\| \leq \gamma$; and, the Christensen
 distance between $C$ and $D$ is defined by
    \begin{equation}\label{d0}
    d_0(C,D) = \inf\{\gamma > 0 \, : \, C \subseteq_{\gamma}
    D \text{ and } D \subseteq_{\gamma} C\}.
    \end{equation}

    Here are some elementary observations regarding the Christensen
    distance, some of which will be used ahead.
    \begin{remark}\label{d0-basics}  Let $A$ be a normed algebra.
      \begin{enumerate}
      \item $d_0(C, D) \leq 1$ for all $C, D \in
        \text{Sub}_A$.
\item $d_0(C, D) = d_0(\overline{C}, D)= d_0(\overline{C},
  \overline{D})$ for all $C, D \in \text{Sub}_A$.
  \item $d_0$ is not a metric on
    $\mathrm{Sub}_A$ (as it  does
    not satisfy the triangle inequality). However, $d_0$ and $d_{KK}$
    are ``equivalent" in the sense that
  \[
  d_0 (C, D) \leq d_{KK}(C, D) \leq 2 d_0(C, D)
  \]
  for all $C, D \in \mathrm{Sub}_A$. (\cite[Remark 2.3]{Ch3})
\item If $C, D \in \mathrm{Sub}_A$ and $C$ is a norm closed proper
  subalgebra of $D$, then $d_0 (C , D ) =1$. (Follows from \cite[Proposition
    2.4]{Ch3} - also compare with \Cref{basic-facts}(3).)
      \end{enumerate}
      \end{remark}

\subsubsection{Christensen's distance between subalgebras of a  tracial von Neumann algebra}\label{dC-section}

For a von Neumann algebra $\mM$, let
\[
*\text{-Sub}^u_{\mM} :=\{ \text{ *-subalgebras of } \mM \text{ containing } 1_{\mM}\} \text{ and }
\]
\[
W^*\text{-}\mathrm{Sub}_{\mM} :=\{ \text{von Neumann subalgebras of }
\mM\ \text{(\,possibly with different unit)}\}.
\]
  Let $\mM$ be a von Neumann
 algebra with a (fixed) faithful normal tracial state $\tau$. Then,
 $\mM$ inherits a natural inner product structure with respect to the
 inner product $\langle x, y \rangle_{\tau} = \tau(y^*x)$, $x, y \in
 \mM$. The corresponding norm on $\mM$ will be denoted by $\|x\|_{\tau}$,
 i.e., $\|x\|_{\tau} = \tau(x^*x)^{1/2}$, $x \in \mM$; and, the Hilbert
 space completion of $\mM$ is denoted by $L^2(\mM, \tau)$.  There is a
 natural embedding of $\mM$ into $B(L^2(\mM, \tau))$ via left
 multiplication and this allows us to consider $\mM$ as a von Neumann
 algebra on $L^2(\mM, \tau)$.

 Further, in order to distinguish the elements of the inner-product
 structure $\mM$ from that of the von Neumann algebra $\mM$, as is
 standard, we shall use the notation $\widehat{x}$ for the elements of the
 inner-product space $\mM$. Thus, $\langle \widehat{x}, \widehat{y} \rangle_\tau
 = \tau(y^* x)$ and $x(\widehat{1}) = \widehat{x}$ for all $x, y \in \mM$.
  
    Recall from \cite{Ch1} that, for $P,\, Q \in \mathrm{Sub}_{\mM}$
    and a scalar $\gamma > 0$, $P \subset_{\gamma} Q$ if for each $x
    \in B_1(P)$ (w.r.t. $\|\cdot\|$), there exists a $ y \in Q$ such
    that $\|\widehat{x}-\widehat{y}\|_{\tau} \leq \gamma$; and, the Christensen distance
    between $P$ and $Q$ is defined by
    \begin{equation}\label{d-tau}
    d_C(P,Q) = \inf\{\gamma > 0 \, |\, P \subset_{\gamma}
    Q\, \text{and}\, Q \subset_{\gamma} P\}.
    \end{equation}  
    Here are some facts related to the Christensen distance $d_C$. 
    \begin{remark}
      With running notations, the following hold:
      \begin{enumerate}
      \item $d_C(P, Q) \leq 1$ for all $P, Q\in \mathrm{Sub}_{\mM}$.
      \item $d_C $ is a complete metric on
  $W^*\text{-}\mathrm{Sub}_{\mM}\,$. (\cite[Theorem, 5.1]{Ch1})

      \item From \Cref{dMT=dC} and \Cref{dMT-SOT} ahead, it  follows that
          \[
          d_C(P, Q) = d_C(\overline{P}^{S.O.T.},
          \overline{Q}^{S.O.T.}) = d_C(P'', Q'')
          \]
                    for all $P, Q \in *\text{-}\mathrm{Sub}_{\mM}^u$.
      \end{enumerate}
      \end{remark}

 Here is  a well known observation which will be essential
for our discussion - see, for instance, \cite[Proposition 2.6.4]{AP}.
\begin{proposition}\label{vNa-ball-complete}
  Let $M$ be a unital $C^{*}$-algebra equipped with a faithful tracial
  state $\tau$. Then, $M$ is a von Neumann algbra  on $L^{2}(M, \tau)$
  iff $\widehat{B_1(M)}$ is complete w.r.t. $\|.\|_{\tau}$.
  
In particular, $M$ is a $W^*$-algebra iff $\widehat{B_1(M)}$ is
complete w.r.t. $\|.\|_{\tau}$.
\end{proposition}

\subsection{Mashood- Taylor distance between subalgebras of a tracial von Neumann algebra}

Given a $II_1$-factor $M$ with a unique faithful normal tracial state
$\tau$ and any two subfactors $P$ and $Q$ of $M$, Mashood and Taylor
(in \cite{MT}) consider the Hausdorff distance
(w.r.t. $\|\cdot\|_{\tau}$) between ${B_1(P) }$ and
${B_1(Q)}$ and prove some nice results related to continuity
of `Jones' index. They also mention (on \cite[Page
  56]{MT}) that this distance gives the same topology on the set of
subfactors of $M$ as the one given by the Christensen's distance $d_C$. We shall
show in this section that  $d_C= d_{MT}$.

Mashood-Taylor  distance can, in fact, be defined in a slightly more
general setup as follows:

For a von Neumann algebra $\mM$ with a faithful normal tracial state $\tau$, for
any pair $P, Q \in \text{Sub}_\mM$, the Mashood-Taylor distance between
$P$ and $Q$ is defined as
\begin{equation}\label{dMT-expresssion}
d_{MT} (P, Q) = d_{H, \|\cdot\|_\tau } (\widehat{B_1(P)}, \widehat{B_1(Q)}),
\end{equation}
where $d_{H, \|\cdot\|_\tau}$ denotes the Hausdorff distance with
respect to the metric induced by the norm $\|\cdot\|_\tau$ and
$\widehat{S} \subseteq L^2(\mM, \tau)$ for $S \subseteq \mM$ (as in
\Cref{dC-section}).

\begin{lemma}\label{dMT-facts}
  Let $\mM$ be a von Neumann algebra with a  faithful normal
  tracial state $\tau$. Then, the following hold:
\begin{enumerate}
\item  $d_{MT}$  is a semi-metric on $\mathrm{Sub}_{\mM}$.

\item $d_{MT}$ is a metric on $W^*\text{-}\mathrm{Sub}_{\mM}$. 

\item \(
d_C(P,Q)\leq d_{MT}(P,Q) 
\)    for all $P, Q \in \text{Sub}_{\mM}$.
 \end{enumerate}
\end{lemma}

\begin{proof}
  (1) follows readily from the definition.\smallskip

  (2): Note that $\widehat{B_1(\mQ)}$ is closed and bounded in $L^2(\mM,
  \tau)$ for every $\mQ \in W^*\text{-Sub}_{\mM}$ by
    \Cref{vNa-ball-complete}. And, it is well known that the Hausdorff distance is a metric on the collection of closed and bounded subsets of a metric space.\smallskip

(3): Let $P, Q \in \text{Sub}_{\mM}$ and $\epsilon >
0$. By the definition of $d_{MT}(P,Q)$, for each $x \in B_{1}(P)$,
there exists a $z \in B_{1}(Q)$ such that
$\|\widehat{x}-\widehat{z}\|_{\tau}\leq d_{MT}(P,Q) + \epsilon$. This implies
that $ P \subset_{d_{MT}(P,Q) + \epsilon} Q$. Similiarly, we have $ Q
\subset_{d_{MT}(P,Q) + \epsilon }P $. Thus, \( d_C(P,Q)\leq
d_{MT}(P,Q) +\epsilon \) for all $\epsilon > 0$ and we are done.
\color{black}
\end{proof}

The following observation is a nice tool to calculate the distance
between two von Neumann subalgebras and will be used ahead on more
than occasion.

\begin{proposition}\label{dMT-SOT}
Let $(\mM, \tau)$ be as in \Cref{dMT-facts}. Then, 
\[
d_{MT}(P,Q) = d_{MT}\Big(P, \overline{Q}^{S.O.T.}\Big) =
d_{MT}\Big(\overline{P}^{S.O.T.}, \overline{Q}^{S.O.T.}\Big) = d_{MT}\Big(P'',
Q''\Big)
\] 
for all $P, Q \in *\text{-}\mathrm{Sub}^u_{\mM}$.
\end{proposition}
\begin{proof}
In order to avoid any possible confusion, for every $x \in \mM$ and $S
\subseteq \mM$, let $d_\tau\big(\widehat{x}, \widehat{S}\big):=\inf \{ \|\widehat{x} -
\widehat{s}\|_\tau : s \in S\}$.

Now, let $P, Q \in *\text{-Sub}^u_{\mM}$.  We first assert that
$d_\tau\Big(\widehat{a}, \widehat{B_1(Q)}\Big) = d_\tau\Big(\widehat{a},
\widehat{B_1(\overline{Q}^{S.O.T.})}\Big)$ for all $a \in B_1(P)$.

Clearly, $d_\tau\Big(\widehat{a}, \widehat{ B_1(\overline{Q}^{S.O.T.}\Big)}\leq
d_\tau\Big(\widehat{a}, \widehat{B_1(Q)}\Big)$ for all $a\in B_1(P)$. For the
reverse inequality, fix an $a \in B_1(P)$ and an $\epsilon>0$. Then,
there exists a $b_{0}\in B_1(\overline{Q}^{S.O.T.})$ such that
\[
d_\tau\Big(\widehat{a},
\widehat{B_1(\overline{Q}^{S.O.T.})}\Big) \leq
\|\widehat{a}-\widehat{b_{0}}\|_{\tau}< d_\tau\Big(\widehat{a},
\widehat{B_1(\overline{Q}^{S.O.T.})}\Big)+ \epsilon.
\]
As $ B_1\Big(\overline{Q}^{S.O.T.}\Big)= \overline{B_1(Q)}^{S.O.T.}$ (by
Kaplansky density theorem), there exists a net
$\{b_{\alpha}\}\subseteq B_1(Q)$ such that
$b_{\alpha}\stackrel{S.O.T.}{\longrightarrow} b_{0}$. Hence, there
exists an $\alpha_{0}$ such that
\[
\|\widehat{a}-\widehat{b}_{\alpha_{0}}\|_{\tau}< d_\tau \Big(\widehat{a},
\widehat{B_1(\overline{Q}^{S.O.T.})}\Big) + \epsilon.
\]
This implies that
\[
d_\tau\big(\widehat{a}, \widehat{B_1(Q)}\big)\leq
\|\widehat{a}-\widehat{b}_{\alpha_{0}}\|_{\tau} <d_\tau \Big(\widehat{a},
\widehat{B_1(\overline{Q}^{S.O.T.})}\Big)+ \epsilon.
\]
As $\epsilon>0$ was arbitrary, we get $d_\tau\Big(\widehat{a},
\widehat{B_1(Q)}\Big)\leq d_\tau\Big(\widehat{a},
\widehat{B_1(\overline{Q}^{S.O.T.})}\Big)$. This proves our assertion.

Then, we get
\[
\beta := \sup_{a\in B_1(P)}d_\tau \big(\widehat{a}, \widehat{B_1(Q)}\big)=
\sup_{a\in B_1(P)}d_\tau \Big(\widehat{a},
\widehat{B_1(\overline{Q}^{S.O.T.})}\Big)\]
\[
\leq \sup_{z\in
  B_1(\overline{P}^{S.O.T.})}d_\tau\Big(\widehat{z},
\widehat{B_1(\overline{Q}^{S.O.T.})}\Big)=: \alpha.
\] 
Further, for any  $\eta>0$, there exists a $ z_{0}\in
B_1\Big(\overline{P}^{S.O.T.}\Big)$ such that
\[
\alpha-\eta<
d_\tau \Big(\widehat{z_{0}}, \widehat{B_1(\overline{Q}^{S.O.T.})}\Big)\leq
\alpha.
\]
Then, by Kaplansky density theorem again, there exists a net
$\{z_{\lambda }\} \subset B_1(P)$ such that $z_{\lambda
}\stackrel{S.O.T.}{\longrightarrow} z_{0}$. In particular,
$\widehat{z_{\lambda}} {\longrightarrow} \widehat{z_{0}}$ with respect
to $\|\cdot\|_\tau$; and, since $d_\tau$ is continuous with respect to
$\|\cdot\|_\tau$, it follows that \( d_\tau \Big(\widehat{{z}_{\lambda }},
\widehat{B_1(\overline{Q}^{S.O.T.})}\Big)\rightarrow
d_\tau\Big(\widehat{z_{0}}, \widehat{B_1(\overline{Q}^{S.O.T.})}\Big).  \)
Thus, there exists a $\lambda_{0}$ such that \( d_\tau
\Big(\widehat{z}_{\lambda_{0}}, \widehat{B_1(\overline{Q}^{S.O.T.})}\Big)>\alpha-
\eta; \) so that, $\beta \geq \alpha-\eta$. Since $\eta > 0$ was
arbitrary, we have $\beta\geq \alpha$. Thus, $\alpha=\beta$, i.e., \(
d_{MT}(P, Q)= d_{MT}\Big({P}, \overline{Q}^{S.O.T.}\Big).  \) By a similar
argument, we conclude that
\[
d_{MT}(P, Q)= d_{MT}\Big(\overline{P}^{S.O.T.}, \overline{Q}^{S.O.T.}\Big).
\]
Rest follows from von Neumann's Double
Commutant theorem.
\end{proof}

\subsection{Comparisons between $d_{KK}$, $d_C$ and $d_{MT}$}

We first present the following comparison between $d_{MT}$ and $d_{KK}$. 
\begin{lemma}\label{dMT-dKK}
  Let $\mM$ be a von Neumann algebra with a faithful normal tracial
  state $\tau$.  Then,
$$d_{MT}(P, Q)\leq d_{KK}(P, Q)$$
  for all $P, Q \in \mathrm{Sub}_{\mM}$.

\end{lemma}

\begin{proof} Let $P, Q \in \mathrm{Sub}_{\mM}$ and 
$p \in B_1(P)$. Then, 
\[
d_\tau \Big(\widehat{p}, \widehat{B_1(Q)}\Big)\leq \|\widehat{p}- \widehat{q}\|_{\tau}\leq \| p- q\| 
\]
for all $ q \in B_1(Q)$. Thus, 
\(
d_\tau \Big(\widehat{p}, \widehat{B_1(Q)}\Big) \leq d\big(p , B_1(Q)\big)
\) for all $p \in B_1(P)$. In particular, 
\[
\sup_{ p \in B_1(P )}d_\tau \Big(\widehat{p}, \widehat{B_1(Q)}\Big)\leq \sup_{ p \in B_1(P )}d\big( p , B_1(Q)\big).\]
Similarly, 
\[
\sup_{q \in B_1(Q )}d_\tau \Big(\widehat{q}, \widehat{B_1(P)}\Big)\leq \sup_{q \in B_1(Q))}d\big(q, B_1(P)\big).
\]
Hence, 
\[
d_{MT}(P , Q)\leq d_{KK}(P, Q).
\]
\end{proof}

The following useful observation  is well known to the experts
(see, for instance, \cite[Eqn.(6)]{Ch1}).

\begin{lemma}\label{x-ENx-relation}
Let $(\mM, \tau)$ be as in \Cref{dMT-dKK} and let
 $\mN$ be a von Neumann subalgebra of $\mM$ (with common unit). Then,
\[
 d_\tau \Big(\widehat{x}, \widehat{B_1(\mN)}\Big)= \|\widehat{x}- \widehat{E_{\mN}(x)}\|_{\tau} =  d_\tau \Big(\widehat{x}, L^{2}(\mN)\Big)
\]
for all $x  \in B_1(\mM)$, where $E_{\mN}$ denotes the unique $\tau$-preserving normal conditional expectation from $\mM$ onto $\mN$.
\end{lemma}

We are now all set to show that the Mashood-Taylor distance agrees with
the Christensen's distance on $*\text{-Sub}_M^u$.
\begin{proposition}\label{dMT=dC}
Let $(\mM, \tau)$ be as in \Cref{x-ENx-relation}. Then,   
$$d_C(P, Q)= d_{MT}(P, Q)$$
for all $P, Q \in *\text{-}\mathrm{Sub}_{\mM}^u$.
\end{proposition}

\begin{proof}
  Let $P, Q \in *\text{-}\mathrm{Sub}_{\mM}^u$ and consider
  $\widetilde{P}:= \overline{P}^{S.O.T.}$ and $\widetilde{Q}:=
  \overline{Q}^{S.O.T.}$.  In view of \Cref{dMT-facts} and
  \Cref{dMT-SOT}, it suffices to show that $d_C(P, Q) \geq
  d_{MT}\big(\widetilde{P}, \widetilde{Q}\big)$.

We first assert that  $d_\tau \Big(\widehat{p},
\widehat{(B_1(\widetilde{Q}))}\Big)= d_\tau \Big(\widehat{p},
L^{2}(\widetilde{Q})\Big)= d_\tau \Big(\widehat{p}, \widehat{Q}\Big)$ for all $p \in
B_1(P)$.

Let $p \in B_1(P)$. Then,  by \Cref{x-ENx-relation}, 
\[
d_\tau \left(\widehat{p}, \widehat{B_1(\widetilde{Q})}\right)= \| \widehat{p} - \widehat{E_{\widetilde{Q}}(p)}\|_\tau = d_\tau \left(\widehat{p},
L^{2}(\widetilde{Q})\right).
\]
Thus, as $\widehat{B_1(Q)}\subset \widehat{Q} \subset
L^2(\widetilde{Q})$, we see that
\[
d_\tau \left(\widehat{p}, \widehat{B_1(\widetilde{Q})}\right)=
d_\tau \left(\widehat{p}, L^{2}(\widetilde{Q})\right)\leq
d_\tau \left(\widehat{p}, \widehat{Q}\right) \leq d_\tau \left(\widehat{p},
\widehat{B_1(Q)}\right)= d_\tau \left(\widehat{p},
\widehat{B_1(\widetilde{Q})}\right)
\]
for all $p \in B_1(P)$. Hence, our first assertion is true.

Now, suppose that $P\subset_{\gamma} Q$ and $Q\subset_{\gamma} P$ for
some $\gamma > 0$. We then assert that
\[
d_\tau \left(\widehat{p}, \widehat{B_1(\widetilde{Q})}\right) \leq \gamma
\text{ for all } p\in B_1(\widetilde{P}).
\]
So, let $p\in B_1(\widetilde{P})$ and,  by Kaplansky density theorem, fix
a net $\{p_{\alpha}\}\subset B_1(P)$ such that
$p_{\alpha}\stackrel{S.O.T.}{\longrightarrow} p$ in $B(L^2(\mM,
\tau))$. Since $P \subset_\gamma Q$, for each
$\alpha$, there exists a $ q_{\alpha}\in Q$ such that $\|\widehat{p}_{\alpha}-
\widehat{q}_{\alpha}\|_{\tau}\leq \gamma.$ So,
\[
\gamma \geq \|\widehat{p}_{\alpha}- \widehat{q}_{\alpha}\|_{\tau}\geq d_\tau
(\widehat{p}_{\alpha}, \widehat{Q})= d_\tau \left(\widehat{p}_{\alpha},
\widehat{B_1(\widetilde{Q})}\right)= \|\widehat{p}_{\alpha}-
\widehat{E_{\widetilde{Q}}(p_{\alpha})} \|_{\tau}
\]
for all $\alpha$, where $E_{\widetilde{Q}}$ is the unique $\tau$-preserving normal
conditional expectation from $\mM$ onto $\widetilde{Q}$.
Thus, 
\[
d_\tau \left(\widehat{p}, \widehat{B_1(\widetilde{Q})}\right)  \leq  \inf_{\alpha} \|\widehat{p}- \widehat{E_{\widetilde{Q}}(p_{\alpha})} \|_{\tau}
 \leq  \inf_{\alpha}\| \widehat{p} - \widehat{p}_{\alpha}\|_{\tau}+ \inf_{\alpha}\|\widehat{p}_{\alpha}- \widehat{E_{\widetilde{Q}}(p_{\alpha})} \|_{\tau}
 \leq  \gamma,
\]
where the last inequality holds because $\|\widehat{p}_{\alpha}-
\widehat{p}\|_{\tau} \rightarrow 0\,$ (as $p_{\alpha} \rightarrow p$ in
S.O.T.) and  $\| \widehat{p_{\alpha}} -
\widehat{E_{\widetilde{Q}}(p_{\alpha})}\|_\tau \leq \gamma$ for all
$\alpha$. Thus, our second assertion is also true.

As a consequence, \( \sup_{p\in B_1(\widetilde{P})}d_\tau \left(\widehat{p},
\widehat{B_1(\widetilde{Q})}\right)\leq \gamma.  \) Likewise, \(
\sup_{q\in B_1(\widetilde{Q})}d_\tau \left(\widehat{q},
\widehat{B_1(\widetilde{P})}\right)\leq \gamma \) as well.  Hence,
$d_{MT}(\widetilde{P}, \widetilde{Q})\leq \gamma$, which then implies that
\[
d_{MT}\big(\widetilde{P}, \widetilde{Q}\big) \leq d_C(P, Q),
\]
and we are done.
\end{proof}

In view of \Cref{dMT-dKK}, we also have the following:
\begin{cor}\label{dC-dKK}
Let $(\mM, \tau)$ be as in \Cref{x-ENx-relation}. Then,   
$$
d_C(P, Q)\leq d_{KK}(P, Q)$$
for all
$P, Q \in \mathrm{Sub}_{\mM}$.
\end{cor}

\section{Kadison-Kastler and Christensen distance between subalgebras corresponding to subgroups}\label{calculations}

\subsection{Distances between crossed-product subalgebras associated to subgroups}
Let $G$ be a discrete group acting on a $C^*$-algebra $\mA$ via the
map $\alpha: G \rightarrow \Aut(\mA) $. Consider the space $C_{c}(G,
\mA)$ consisting of compactly supported $\mA$-valued functions on $G$,
which can be identified naturally with the vector space
$\{\sum_{\text{finite}}a_{g}g\,\,:\,\,a_{g}\in \mA\,,\,g\in G\}$
consisting of formal finite sums. Further, $C_{c}(G, \mA)$ is a
$*$-algebra with respect to the (so called, twisted) multiplication
given by the convolution operation
\[
\Big(\sum_{s\in I}a_{s}s\Big)\Big(\sum_{t\in J}b_{t}t\Big)= \sum_{s\in
  I, t\in J}a_{s}\alpha_{s}(b_{t})st
\]
and involution given by 
\[
\Big(\sum_{s\in I}a_{s}s\Big)^{*}= \sum_{s\in I}\alpha_{s^{-1}}(a_{s}^{*})s^{-1}
\]
for any two finite sets $I$ and $J$ in $G$. The reduced crossed
product $\mA \rtimes^r_{\alpha}G$ and the universal crossed product
$\mA \rtimes^u_{\alpha} G$ are defined, respectively, as the completions of
$C_{c}(G, \mA)$ with respect to the reduced norm and the universal norm
on $C_c(G, \mA)$, as described below:

The reduced norm is given by 
\[
\Big\|\sum_{\text{finite}}a_{g}g\Big\|_r:= \Big\|\sum_{\text{finite}}\pi(a_g)(1\otimes
\lambda_g)\Big\|_{B(H\otimes l^2(G))},
\]  
where $\mA \subset B(H) $ is a (equivalently, any) fixed faithful
representation of $\mA$, $\lambda: G\rightarrow B(l^2(G))$ is the left
regular representation and $\pi: \mA \rightarrow B(H\otimes l^2(G))$ is
the representation satisfying $\pi(a)(\xi \otimes \delta_g)=
\alpha_{g^{-1}}(a)(\xi)\otimes \delta_g$ for all $\xi \in H$ and $g\in
G$.

The universal norm is given by 
\[
\|x\|_u:= \sup_{\pi}\|\pi(x)\|_{B(H_{\pi})}\, \mbox{\,\,for\,\,}\,x\in C_c(G, \mA),
\]
where the supremum runs over all (cyclic) $*$-homomorphisms $\pi:
C_c(G, \mA)\rightarrow B(H_{\pi})$. Note that 
$
\|x\|_{r}\leq \|x\|_u \ \text{ for all}\  x\in C_c(G, \mA).
$

We refer the reader to \cite{BO} for more on crossed
products.

\subsubsection{Distances between subalgebras of reduced crossed product}
First,  we  gather some relevant facts related to the
reduced crossed product construction.
\begin{remark}\label{crossed-product-facts}
Let $G$ and $\mA$ be as above and $H$ be a subgroup of $G$. 
\begin{enumerate}
\item The canonical injective $*$-homomorphism 
\[
C_c(H, \mA) \ni \sum_{\mathrm{finite}} a_h h \mapsto
\sum_{\mathrm{finite}} a_h h\in C_c(G, \mA)
\]
 extends to an injective $*$-homomorphism from $\mA \rtimes^r_{\alpha} H
 $ into $\mA \rtimes^r_{\alpha} G$. Hence, we can consider
 $\mA \rtimes^r_{\alpha} H$ as $C^*$-subalgebra of
 $\mA \rtimes^r_{\alpha}G$. (\cite[Remark 3.2]{MK})
\item There exists a faithful conditional expectation $E_H: \mA
  \rtimes_{\alpha}^rG \rightarrow \mA \rtimes_{\alpha}^r H$ satisfying
  $E_H(\sum_{g \in I} a_{g}g) = \sum_{h\in I \cap H} a_h h$ for all $
  \sum_{g \in I}a_{g}g\in C_c(G, \mA)$ (\cite[Remark 3.2]{MK}); and,
  if $[G:H] < \infty$, then $E_H$ has finite (Watatani) index with a
  quasi-basis given by any set of left coset representatives $\{g_i: 1
  \leq i \leq [G:H]\}$ of $H$ in $G$.

\item Thus, as in \Cref{index-basic-construction}, the vector
  space $\mA\rtimes^r_{\alpha}G$ is a natural left $(\mA
  \rtimes^r_{\alpha} H)$-module via multiplication from left. Hence,
  following \cite{Wat}, we consider $\mA \rtimes^r_{\alpha}G$ as a
  pre-Hilbert (left) $(\mA \rtimes^r_{\alpha} H)$-module with respect
  to the $(\mA \rtimes^r_{\alpha} H)$-valued inner product given by
\[
\langle x , y \rangle = E_H(x^*y), \text{ for } x, y \in \mA \rtimes^r_{\alpha}G.
\]
Further, in order to distinguish the elements of the $C^*$-algebra
$\mA \rtimes^r_{\alpha}G$ and the pre-Hilbert $(\mA \rtimes^r_{\alpha}
H)$-module $\mA \rtimes^r_{\alpha}G$, following \cite{Wat}, we
consider the identity map $\eta_H : \mA \rtimes^r_{\alpha}G
\rightarrow \mA \rtimes^r_{\alpha}G$, where we consider the co-domain
as the pre-Hilbert $(\mA \rtimes^r_{\alpha}H)$-module.  Since $E_H$ is
a contraction, we have
 \begin{equation}\label{r-eta-relation}
\|\eta_H(x)\|  = \|E_H(x^*x)\|^{1/2}_r \leq \|x\|_r
  \end{equation}
  for all $ x \in \mA \rtimes^r_{\alpha} G$.
\item Note that, for two subgroups $H \subsetneq K$ of a discrete
  group $ G$, $\|\eta_H(x)\|$ need not be equal to $\|\eta_K(x)\|$ for
  every $x \in \mA \rtimes^r_{\alpha} G$.
  
  For instance, consider the trivial action of the finite permutation
  group $G=S_{3}$ on the $C^*$-algebra $\mathbb{C}$. Let $H=\{e\}$,
  $K=A_{3}$ and $x=a_{1}(123)+a_{2}(132)\in C_{r}^*(G) = \C[G]$ with
  $a_1,a_2\neq 0$. Then, $x^*x= (|a_{1}|^2+|a_{2}|^2)e+
  \bar{a_1}a_2(123)+\bar{a_2}a_1(132)$ and 
 \begin{eqnarray*}
   \|\eta_K(x)\|^{2} &=& \|E_K(x^*x)\|_r\\ & = &
   \Big\|(|a_{1}|^2+|a_{2}|^2){e}+ (\bar{a_1}a_2)\,
         {(123)}+(\bar{a_2}a_1)\, {(132)}\Big\|_r \\ & = &
         \Big\|(|a_{1}|^2+|a_{2}|^2)\lambda_{e}+ (\bar{a_1}a_2)\,
         \lambda_{(123)}+(\bar{a_2}a_1)\,
         \lambda_{(132)}\Big\|_{B(\ell^2(S_3))}\\ &\geq &
         \Big\|\Big((|a_{1}|^2+|a_{2}|^2)\lambda_{e}+ (\bar{a_1}a_2)\,
         \lambda_{(123)}+(\bar{a_2}a_1)\,
         \lambda_{(132)}\Big)(\delta_{e})\Big\|_{\ell^2(S_3)}\\ &=&
         \Big\|(|a_{1}|^2+|a_{2}|^2)\delta_{e}+
         (\bar{a_1}a_2)\,\delta_{(123)}+(\bar{a_2}a_1)\,\delta_{(132)}
         \Big\|_{\ell^2(S_3)} \\ & =& \Big((|a_{1}|^2+|a_{2}|^2)^2+
         |\bar{a_1}a_2|^2+ |\bar{a_2}a_1|^2\Big)^{\frac{1}{2}}\\ &>&
         \Big((|a_{1}|^2+|a_{2}|^2)^2\Big)^{\frac{1}{2}}\\ &=&
         (|a_{1}|^2+|a_{2}|^2) \\ & = & \|\eta_H(x)\|^2 .
\end{eqnarray*}    
\item We shall denote $E_{\{e\}}$ (resp., $\eta_{\{e\}}$) simply by
  $E$ (resp., $\eta$).
\end{enumerate}
\end{remark}

The following is an elementary and useful observation. A more general
version of it was proved by Phillips (see \cite[Proposition 9.16(3)]{N}).
\begin{lemma}\label{ag-inequality}
Let $G$ be a discrete group acting on a $C^*$-algebra $\mA$. If $a=\sum_{g \in I}
a_{g}g\in C_c(G, \mA)$, then
\[
\|a_g\|^2\leq \|E(a^*a)\|_r=\|\eta(a)\|^2 \text{ for all }\,g\in I.
\]
\end{lemma}

\begin{proposition}\label{Cc-HA-Cc-KA-distance}\label{d0-distance}
Let $G$, $\mA$ be as in \Cref{ag-inequality} and let  $H$ and $K$ be two
distinct subgroups of $G$. Then,
  \[
  d_{KK}(C_c(H, \mA),C_c(K, \mA))=1 = d_0(C_c(H, \mA),C_c(K, \mA))
  \]
in $\mA \rtimes^r_{\alpha} G$.
\end{proposition}
\begin{proof}
Note that, $d_{KK}(C_c(H, \mA),C_c(K, \mA))\leq 1$, by definition and
$d_{0}(C_c(H, \mA), C_c(K, \mA)) \leq d_{KK}(C_c(H, \mA),C_c(K,
\mA))$, by \Cref{d0-basics}. So, it just remains  to show that
\[ d_{0}(C_c(H, \mA),C_c(K, \mA)) \geq 1.
\]
Since $H$ and $K$ are distinct, either $H \neq H \cap K$
or $K \neq H \cap K$.  Without loss of generality, we can assume that
$H\neq H\cap K$. Then, in view of
(\ref{r-eta-relation}) and \Cref{ag-inequality}, we observe that
\[
\| h - x\|_r \geq \|\eta(h - x)\| \geq 1 >\gamma
\]
for all
$h \in H \setminus H \cap K$,  $x \in B_1(C_c(K,
\mA))$ and for every $ 0 < \gamma < 1$. Thus,
  if $C_c(H, \mA) \subseteq_{\gamma} C_c(K, \mA)$ for some $\gamma >
  0$, then $\gamma \geq 1$. So, by the definition of $d_0$, we must have
  \[
  d_{0}(C_c(H, \mA), C_c(K, \mA)) \geq 1.
  \]
\end{proof}

\subsubsection{Distances between subalgebras of universal crossed product}
 
\begin{remark}
Let $G$, $\mA$ and $\alpha: G\rightarrow \mathrm{Aut}(\mA)$ be as in the
preceding subsection and $H$ be a subgroup of $G$. Then, the canonical
injective $*$-homomorphism
\[
C_c(H, \mA) \ni \sum_{\mathrm{finite}} a_h h \mapsto \sum_{\mathrm{finite}} a_h h\in C_c(G, \mA)
\]
 extends to an injective $*$-homomorphism from $\mA \rtimes^u_{\alpha} H
 $ into $\mA \rtimes^u_{\alpha} G$. Hence, we can consider
 $\mA \rtimes^u_{\alpha} H$ as $C^*$-subalgebra of
 $\mA \rtimes^u_{\alpha}G$. (\cite[Proposition 3.1]{MK})
 \end{remark}

Since $\|x\|_u \geq \|x\|_r$ on $C_c(G, \mA)$, the following is immediate from the proof of \Cref{Cc-HA-Cc-KA-distance}.
\begin{proposition}\label{Cc-HA-Cc-KA-distance-universal}\label{d0-distance-universal}
Let $G$, $\mA$ be as in \Cref{ag-inequality} and let $H$ and $K$ be
two distinct subgroups of $G$. Then,
  \[
  d_{KK}(C_c(H, \mA),C_c(K, \mA))=1 =   d_{0}(C_c(H, \mA),C_c(K, \mA))
  \]
in $\mA \rtimes^u_{\alpha} G$.
\end{proposition}

In view of \Cref{dense-distance}, its analogue in \Cref{d0-basics} and
the preceding two propositions, we obtain:
\begin{cor}\label{distances-crossed-product-subalgebras}
  Let $G, H, K$, $\mA$ and $\alpha$ be as in \Cref{Cc-HA-Cc-KA-distance}. Then,
\begin{enumerate}
\item $d_{KK}(A\rtimes^r_{\alpha}H, A\rtimes^r_{\alpha}K)=1 = d_{0}(A\rtimes^r_{\alpha}H, A\rtimes^r_{\alpha}K)$ in $\mA \rtimes^r G$; and,
\item $d_{KK}(A\rtimes^u_{\alpha}H, A\rtimes^u_{\alpha}K)=1 = d_{0}(A\rtimes^u_{\alpha}H, A\rtimes^u_{\alpha}K) $ in  $A\rtimes^u_{\alpha}G$.
\end{enumerate}
\end{cor}

\subsubsection{Some observations related to the $C^*$-algebras associated to groups}

Recall that for any discrete group $G$, if $\mA =\C$ and $\alpha: G
\to \mathrm{Aut}(G)$ is the trivial action, then $\mA
\rtimes^r_{\alpha} G$ (resp., $\mA \rtimes^u_r G$) is just the reduced
group $C^*$-algebra $C^*_r(G)$ (resp., the (universal) group
$C^*$-algebra $C^*_u(G)$). Thus, we readily obtain the following:

\begin{cor}\label{d-CH-CK}\label{dKK-d0-group-algebras}
Let $G$ be a discrete group and $H$ and $K$ be two distinct subgroups of $G$. Then,
\begin{enumerate}
\item $d_{KK}(\C[H], \C[K]) = d_{KK}(C^*_r(H), C^*_r(K)) = 1 =  d_{0}(\C[H], \C[K]) = d_{0}(C^*_r(H), C^*_r(K)) $ in $C^*_r(G)$; and, 
\item $d_{KK}(\C[H], \C[K]) = d_{KK}(C^*_u(H), C^*_u(K)) = 1 = d_{0}(\C[H], \C[K]) = d_{0}(C^*_u(H), C^*_u(K)) $ in $C^*_u(G)$. 
\end{enumerate}
\end{cor}

As in \Cref{crossed-product-facts}, for any subgroup $H$ of $G$, we
consider the identity map $\eta_H: C^*_r(G) \to C^*_r(G)$ with the
pre-Hilbert $C^*_r(H)$-norm $\|\eta_H(x)\|:= \|E_H(x^*x)\|_r^{1/2}$
for all $x \in \C[G]$. Also, we simply write $\eta$ for $ \eta_{e}$.

\begin{lemma}\label{ce-on-CG}
With running notations, 
\begin{equation}
\left\|\eta\left(\sum_{g \in G} \alpha_g g \right)\right\|^2\leq \sum_{g
   \in G} |\alpha_g| \left( \sum_{x \in gH}|\alpha_x|\right)
 \end{equation}
 for all
 $\sum_g \alpha_g g \in \C[G]$.
\end{lemma}
\begin{proof}
  Note that
  \[
  E\left(\Big(\sum_{g \in G}
  \alpha_g g\Big)^*\Big(\sum_{g \in G} \alpha_g g\Big)\right)= \sum_{g
    \in G} \sum_{x \in gH} \bar{\alpha_g}\alpha_x g^{-1}x =
  \sum_g \bar{\alpha_g}g^{-1} \Big(\sum_{x \in g H} \alpha_x x\Big).
  \]
 for all
 $\sum_g \alpha_g g \in \C[G]$.
\end{proof}

Further, the reduced group $C^*$-algebra always admits a canonical
faithful tracial state $\tau: C^*_r(G) \to \C$ which satisfies
$\tau(x) = \langle \lambda(x) \delta_e , \delta_e \rangle$ for all $x
\in \C[G]$, where $\lambda: \C[G] \to B(\ell^2(G))$ is the faithful $*$-representation induced by the (left)
regular representation of $\C[G]$.

\begin{remark}  \label{H-trivial-case}
Let $G$ be a  discrete group.
\begin{enumerate}
\item If $H$ is the trivial subgroup of $G$, then the conditional
  expectation $E$ from $C^*_r(G)$ onto $\mbb{C}\ (\equiv C^*_r(H))$ is
  just the tracial state $\tau$ on $C^*_r(G)$, which satisfies $
  E(\sum_{g \in G}\alpha_{g}g)=\alpha_{e} $ for all $\sum_g \alpha_g g
  \in \C[G]$, the $\C$-valued inner product induced by $E$ is just the
  usual inner product on $C^*_r(G)$ induced by $\tau$, and
  \begin{equation}\label{eta-formula}
  \left\|\eta(\sum_{g \in G} \alpha_g g )\right\|=\sqrt{\sum_{g \in
      G}\left|\alpha_g\right|^2}
  \end{equation}
  for all $\sum_g \alpha_g g \in
  \C[G]$ because \[ E\left((\sum_{g \in G} \alpha_g g)^*(\sum_{g \in
    G} \alpha_g g)\right)=\sum_{g \in G}\left|\alpha_g\right|^2.
  \]
\item We thus deduce that $\|\sum_g \alpha_g g \|_r \geq
  \sqrt{\Big(\sum_g|\alpha_g|^2)}$ for all $\sum_g \alpha_g g \in
  \C[G]$.

In particular, $\sqrt{2} \leq \|g -h \|_r \leq {2}$ for any two
distinct elements $g , h$ of $ G$. Thus, $\{ g : g \in G\}$ is a
discrete linearly independent subset of the unit sphere of $C^*_r(G)$.

\item If $G$ is finite, then $C^*_r(G) =\C[G]$ and
\begin{equation}\label{norm-equivalence}
 \|\eta(x) \| \leq \|x\|_r \leq |G|\, \|\eta(x)\|
  \end{equation}
  \text{ for all } $ x \in \C[G]$, because, for every $x = \sum_g
  \alpha_g g \in \C[G]$, $\|\eta (x) \| = \sqrt{\sum_{g \in
      G}\left|\alpha_g\right|^2}$ (and the last inequality in
  (\ref{norm-equivalence}) follows from the H\"older's inequality).

Moreover, $E: \C[G] \to \C$ has finite index with a quasi-basis $\{g:
g \in G\}$ and $\mathrm{Ind}_W(E) = |G|$. 
  \end{enumerate}
\end{remark}

\begin{remark}
If $\mB$ is a $*$-subalgebra of a unital $C^*$-algebra $\mA$, then
$\mathcal{N}_{\mA}(\mB)$ is a subgroup of
$\mathcal{N}_{\mA}(\overline{\mB})$. In particular, for any subgroup
$H$ of a discrete group $G$,
\[
\mathcal{N}_G(H) \leq \mathcal{N}_{C^*_r(G)}(\C[H]) \leq \mathcal{N}_{C^*_r(G)}(C^*_r(H)). 
\]
\end{remark}

Given a subgroup $H$ of a discrete group $G$ and a unitary $u$ in
$\mU(C^*_r(G))\setminus \mcal{N}_{\C[G]}(\C[H])$, it is natural to ask
whether $u \C[H] u^* = \C[K]$ or $u C^*_r(H) u^* = C^*_r(K)$ for some
subgroup $K$ of $G$ or not. Obviously, for every $g \in G$, $g \C[H]
g^* = \C[gHg^{-1}]$, $g C^*_r(H) g^* = C^*_r(gHg^{-1})$ and $\|g -
\1\|_r \geq \sqrt{2}$.  However, when $u\notin G$, then it is not
clear when $u\C[H] u^*$ equals $\C[K]$ for some subgroup $K$ of
$G$. \Cref{dKK-d0-group-algebras} allows us to deduce the following
partial answer to this question.

\begin{proposition}\label{CH-conjugate}
Let $H$ be a  proper subgroup of a  discrete group $G$ and $u$
be a unitary in $\C[G]$ (resp., $C^*_r(G)$) such that $\|u - \1\|_r <
1/2$. If $u\mbb{C}[H]u^{*} = \mbb{C}[K]$ (resp., $ uC^*_r(H) u^* =
C^*_r(K)$) for some subgroup $K$ of $G$, then $K = H$ and, in particular, 
$u \in \mathcal{N}_{\C[G]}(\C[H])$ (resp., $u \in
\mathcal{N}_{C^*_r(G))}(C^*_r(H))$), 
\end{proposition}

\begin{proof}
We first discuss the group algebra case. Let $u \in \C[G]$ with $\|u -
\1\|_r < 1/2$.

Suppose that $u\mbb{C}[H]u^{*} = \mbb{C}[K]$ for some subgroup $K$ of
$G$. Then, by \Cref{u-1-relation}, we observe that
\[
d_{KK}(\C[H], \C[K]) =  d_{KK}(\C[H], u \C[H] u^*) \leq 2 \|u - \1\|_r < 1.
\]
Thus,  $K = H$, by \Cref{d-CH-CK}. 

Even in the reduced group $C^*$-algebra case, when $u$ is a unitary in
$C^*_r(G)$ satisfying the inequality $\|u - \1\|_r < 1/2$ and that $
uC^*_r(H) u^* = C^*_r(K)$ for some subgroup $K$ of $G$, the same
argument shows that $ d_{KK}( C^*_r(H), C^*_r(K))<1$; so that, $K = H$
(by \Cref{d-CH-CK} again).
\end{proof}

Here are two obvious reformulations of the preceding corollary.
\begin{remark}
  Let $H$, $G$ be as in \Cref{CH-conjugate} and
  $u$ be a unitary in $\C[G]$ (resp., $C^*_r(G)$).
\begin{enumerate}
\item If $u \C[H] u^* = \C[K]$ (resp., $u C^*_r(H) u^* = C^*_r(K)$)
  for some subgroup $K$ other than $H$, then $\|u - \1\|_r \geq 1/2$.
\item   If $\|u - \1\|_r< 1/2$, then the conjugate $*$-subalgebra $u
  \C[H] u^*$ (resp., $u C^*_r(H) u^*$) is not equal to $\C[K]$ (resp.,
  $C^*_r(K)$) for any subgroup $K$ other than $H$. 
\end{enumerate}
\end{remark}

\begin{remark}
One could ask whether every unitary $u $ in
$\mathcal{N}_{\C[G]}(\C[H])$ (resp., in
$\mathcal{N}_{C^*_r(G)}(C^*_r(H))$) satisfies the inequality $\|u -
\1\|_r < 1/2$. This is trivially seen to be false.

Indeed, if $H$ is a subgroup of $G$ with a non-trivial normalizer,
then for any $e \neq g \in \mathcal{N}_G(H)$, $g \C[H]g^* = \C [g
  Hg^{-1}] = \C[H]$ whereas, as noted in \Cref{H-trivial-case}(2), we have
\[
\|g - \1\|_r = \|g - e\|_r \geq \sqrt{2} > 1/2.
\]
By the same argument, one also concludes that if $u \in
\mathcal{N}_{C^*_r(G)}(C^*_r(H))$, then the inequality $\|u - \1\|_r<
1/2$ need not be true.
\end{remark}

\begin{remark} 
Note that, by \Cref{close-unitary}, we can always find a unitary $u$
in $C^*_r(G)$ such that $0< \|\1- u\|_r< \frac{1}{2}$. Hence, as is well
known, not every $C^{*}$-subalgebra of $C^*_r(G)$ is a reduced
subgroup $C^*$-algebra.
\end{remark}
The following recipe provides a concrete way of obtaining unitaries
arbitrarily close to $\1$.
\begin{lemma}\label{u-theta}
Let $G$ be a discrete group. If there exists an element $g \in G$ such
that $g=g^{-1}$, then $u_{\theta}:=\cos(\theta) e + i \sin (\theta) g$
is a unitary in $\C[G]$ for every $\theta \in \R$. Also, for each
$\epsilon > 0$, there exists a $\delta > 0$ such that $\|u_\theta -
\1\|_r < \epsilon$ whenever $|\theta| < \delta$.
\end{lemma}
\begin{proof}
Let $\theta \in \R$. Then,
\begin{align*}
u_{\theta}u_{\theta}^{*}&= (\cos(\theta) e + i \sin (\theta) g)
(\cos(\theta) e - i \sin (\theta) g)\\ &=(\cos^2(\theta) + \sin^2
(\theta)) e + (i \sin(\theta)\cos(\theta)- i\sin(\theta)\cos(\theta))
g \\ &= e.
\end{align*} 
 Similiarly, $u_{\theta}^{*}u_{\theta}= e$.  Hence, $u_{\theta}$ is a
 unitary in $\C[G]$ for every $\theta \in \R$.

Note that, for any given $\epsilon > 0$, there exists a $\delta > 0$
such that $|\cos(\theta) - 1|, |\sin (\theta)|< \epsilon/2$ whenever
$|\theta| < \delta$. So,
 \[
 \|u_\theta -\1 \|_r = \|(\cos(\theta)-1)e + i \sin (\theta) g\|_r  \leq |\cos(\theta)-1| + |\sin(\theta)| < \epsilon
 \]
whenever $|\theta| < \delta$. 
\end{proof}

We can now find unitaries arbitrarily close to $\1$ yet not in the
normalizer of a subalgebra. Compare with \Cref{u-in-normalizer}.

\begin{lemma}\label{u-theta-normalizer}
Let $H$ be a proper subgroup of a discrete group $G$. If there exists
an element $g \in G \setminus H$ such that $g= g^{-1}$ and $g \notin
\mC_{G}(H)$ (centralizer of $H$ in $G$), then $u_{\theta}\notin
\mathcal{N}_{C^*_r(G)}(C^*_r(H))$ for every $\theta \in\R \setminus
\{\frac{n\pi}{2}: n\in \Z\}$.
\end{lemma}

\begin{proof}
Since $g \notin \mC_{G}(H)$, there exists an $h \in H$ such that $gh\neq hg$;  also $gh, hg \notin H$. Then, for this $h \in \C[H]$ and $\theta \in\R \setminus
\{\frac{n\pi}{2}: n\in \Z\}$, we have
\[
u_{\theta}hu_{\theta}^{*}= \cos^{2}(\theta)h
+i \sin (\theta)\cos (\theta)gh- i \sin (\theta) \cos (\theta)hg+
\sin^{2} (\theta)ghg.
\]
Hence, $u_{\theta}\notin
\mathcal{N}_{C^*_r(G)}(C^*_r(H))$ for every $\theta \in\R \setminus
\{\frac{n\pi}{2}: n\in \Z\}$.
\end{proof}

\begin{example}\label{u-theta-S3}
Let $G = S_3$ and $H = A_3$.  Then, $u_{\theta}: = \cos(\theta) e + i
\sin (\theta) (12)$ is a unitary in $\C[S_3]$ for every $\theta \in
\R$. Clearly, by the preceding lemma, $u_\theta \notin
\mathcal{N}_{\C[S_3]}(\C[A_3])$ for every $\theta \in\R \setminus
\{\frac{n\pi}{2}: n\in \Z\}$, because $(12) \notin
\mC_{S_3}(A_3)$. Also, we can choose $0 < \theta < \pi/2$ small enough
so that $\|u_\theta - \1\|_r$ is as small as we wish. Thus,
\begin{enumerate}
\item  for each $\epsilon > 0$, there exists a unitary $u$ in $\C[S_3]$ such that
\[
0< d_{KK}(\C[A_3], u \C[A_3] u^*) < \epsilon; \text{and,}
\]
\item $\1$ is not an interior point of
  $\mathcal{N}_{\C[S_3]}(\C[A_3])$ in $\mathcal{U}(\C[S_3])$.
\end{enumerate}
\end{example}

\subsection{Kadison-Kastler distance between (Banach) subgroup algebras}
Recall that, for a discrete group $G$, the Banach space
\[
\ell^1(G):= \Big\{f : G\rightarrow \C \,\,:\,\, \sum_{g \in G}|f(g)| <
\infty \Big\}
\]
is a unital Banach $*$-algebra with multiplication  given by
convolution, i.e., for $a, b \in \ell^1(G),$
$$(ab)(g):= \sum_{h \in G}a(h)b(h^{-1}g);$$
and, involution  given by
$$ a^{*}(g):= \overline{a(g^{-1})}$$ for $a \in \ell^1(G), g \in G.$

Further, for each $g \in G,$ we define $u_{g}\in \ell^1(G)$ by
$u_{g}(h)=\delta_{g, h}$. Then, $u_e$ is the multiplicative identity
for $\ell^1(G)$. 
\begin{remark} With running notations,
  \begin{enumerate}
    \item
There exists an injective
unital $^*$-homomorphism $\mathit{i}: \C[G] \ra \ell^1(G)$ such that
$i(g) = u_g$ for all $g \in G$ and the image of $i$ is dense in $\ell^1(G)$.

In particular, $\C[G]$ can be considered as a dense subspace of $\ell^1(G)$.
\item For any subgroup $H$ of $G$, the natural embedding of  $\C[H]$ into $\C[G]$ extends  to an isometric
unital $*$-homomorphism from $\ell^1(H)$ into $\ell^1(G)$. We can thus
identify $\ell^1(H)$ with a unital Banach $*$-subalgebra of
$\ell^1(G)$.
\end{enumerate}
  \end{remark}
 
\begin{proposition}\label{distance-Banach-gp-algebras}
  Let $H$ and $K$ be two distinct subgroups of a discrete group $G$. Then,
  \begin{equation}
d_{KK}\big(\C[H], \C[K]\big)=    d_{KK}\Big(\ell^1(H), \ell^1(K)\Big)=1
  \end{equation}
  in $\ell^1(G)$.
\end{proposition}
\begin{proof}
Since $\mbb{C}[G]$ is dense in $\ell^1(G)$, by \Cref{dense-distance},
it is enough to show that $d_{KK}(\mbb{C}[H]),\mbb{C}[K])=1$ in
$\ell^1(G)$. In view of \Cref{basic-facts}(3), we can assume that
$H\neq H\cap K\neq K$.

For convenience, let $C:= \C[H]$ and $D := \C[K]$ and let $h\in H\setminus H\cap K$. Then,
\[
\|h- \sum_{k \in K} \alpha_k k\|_{1}\geq 1
\] 
for all $\sum_{k \in K} \alpha_k k \in
B_1(D)$.
This implies that, $d(h, B_1(D))\geq 1$. Thus, by definition, we get 
\[
d_{KK}(\mbb{C}[H]),\mbb{C}[K])\geq 1,
\]
and we are done.
  \end{proof}

\subsection{Kadison-Kastler and Christensen distances between subgroup  von Neumann subalgebras}
Recall that for any discrete group $G$, the group von Neumann algebra
associated to $G$ is the von Neumann algebra given by
\[
L(G) =
\{\lambda(G)\}'' \subseteq B(\ell^2(G)),
\]
where $\lambda: G \to B(\ell^2(G) ) $ is the left regular (unitary)
representation of $G$. Also, there is a natural $*$-isomorhphism
between $\C[G]$ and $*$-$\text{alg}(\lambda(G))$; and, for any subgroup
$H$ of $G$, there is a natural $*$-isomorphism from $L(H)$ onto
$\lambda(H)''\subset L(G)$; thus,
$L(H)$ can be considered as a von Neumann subalgebra of
$L(G)$. Further, $L(G)$ always admits a faithful normal tracial state
$\tau$ given by $\tau(x) = \langle x(\delta_e), \delta_e\rangle$ for
$x \in L(G)$. Thus, $L(G)$ admits an inner-product structure via
$\tau$ which induces a norm $\|\cdot \|_\tau$ on $L(G)$ given by
$\|x\|_\tau = \tau(x^*x)^{1/2}$, $x \in L(G)$.

\begin{proposition}\label{distances-gp-vNas}
Let $G$ be a discrete group and $H$ and $K$ be two distinct
non-trivial subgroups of $ G$.  Then, in $L(G)$, the distances
\[
d_C\big(\mathbb{C}[H], \mathbb{C}[K]\big), d_{MT} \big(\mathbb{C}[H],
\mathbb{C}[K]\big), d_{KK}\big(\mathbb{C}[H], \mathbb{C}[K]\big),
\]
\[
d_C\big(L(H), L(K)\big), 
d_{MT}\big(L(H), L(K)\big)\text{ and } d_{KK}\big(L(H), L(K)\big)
\]
are all equal  to $1$.
\end{proposition}

\begin{proof}
By \Cref{dMT-dKK}, \( d_{MT} \big(\mathbb{C}[H],  \mathbb{C}[K]\big) \leq d_{KK}
\big(\mathbb{C}[H],  \mathbb{C}[K]\big) \leq 1\). So, in view of \Cref{dMT-SOT}
and \Cref{dMT=dC}, it is enough to show that \( d_{MT} \big(\mathbb{C}[H], 
\mathbb{C}[K]\big) \geq 1\).

We prove this by considering the following two
possibilities separately.\smallskip

\noindent \textbf{Case 1.} Suppose that  $H\subsetneq K$ or $K\subsetneq H$.

Without loss of generality, assume that $H \subsetneq K$. Then,
$\mathbb{C}[H]\subsetneq \mathbb{C}[K]$ and \( d_{\tau}\Big(\widehat{a},
\widehat{B_{1}(\mbb{C}[K])}\Big)=0 \) for all $a\in
B_{1}(\mbb{C}[H])$. Further, for any $k
\in K\setminus H \cap K$, we have, \( \|\widehat{k}- \widehat{x} \|_{\tau}\geq 1 \) for all $x \in B_{1}(\mbb{C}[H])$. Thus,
\[
\sup_{z \in B_1(\C[K]))}d_{\tau}\Big(\widehat{z}, \widehat{B_{1}(\mbb{C}[H])}\Big)\geq 1;
\] 
and, hence, $d_{MT}(\mathbb{C}[H], \mathbb{C}[K])\geq 1$. \smallskip

\noindent \textbf{Case 2.} Suppose that $H \neq H\cap K \neq K$.

Then, again for any $h\in H\setminus H\cap K $, we have  
\(
\|\widehat{h}-\widehat{x}\|_{\tau}\geq 1
\) for all $x\in B_{1}(\mbb{C}[K])$.
Thus, as above, $d_{MT}(\mathbb{C}[H], \mathbb{C}[K])\geq 1$, and we are done.
\end{proof}


\begin{thebibliography}{ABC99}
\bibitem{AP} C.~Anantharaman and S.~Popa: An introduction to $II_1$ factors, preprint.

\bibitem{BE} K.~C.~Bakshi, S.~Guin and D.~Jana, `A few remarks on intermediate subalgebras of an inclusion of $C^*$-algebras', arXiv:2311.07244.

  
\bibitem{BG} K.~C.~Bakshi and V.~P.~Gupta, `Lattice of intermediate subalgebras', J. London Math. Soc. \textbf{104}(2)(2021), 2082-2127.


\bibitem{BO} N.~P.~Brown and N.~Ozawa, $C^*$-Algebras and Finite Dimensional Approximations, GSM 88, American Mathematical Society, 2008.  



\bibitem{Ch1} E.~Christensen, `Subalgebras of a finite algebras', Math. Ann. {\bf 243} (1979), 17-29. 

\bibitem{Ch2} E.~Christensen, `Near inclusions of $C^*$-algebras', Acta Math. \textbf{144} (1980), 249-265. 

\bibitem{Ch3} E.~Christensen, A.~M.~Sinclair, R.~R.~Smith, S.~A.~White
  and W.~Winter, `Perturbations of nuclear $C^*$- algebras', Acta
  Math. \textbf{208} (2012), 93-150.

\bibitem{Dickson} L.~Dickson, `A Kadison Kastler row metric and intermediate subalgebras', Internat. J. Math. \textbf{25} (2014), 16 pp.
  
\bibitem{GS} V.~P.~Gupta and D.~Sharma, `On possible values of the interior angle between intermediate subalgebras', J. Aust. Math. Soc. (2023), 1-23.


\bibitem{IN} S.~Ino, `Perturbations of crossed product $C^*$-algebras
  by amenable groups', Kyushu Univer. Instit. Repository, 2015. https://doi.org/10.15017/1654666

  
\bibitem{IW} S.~Ino and Y.~Watatani, ‘Perturbation of intermediate
  $C^{*}$-subalgebras for simple $C^{*}$–algebras’, Bull.
  Lond. Math. Soc. \textbf{46} (2014), 469–480.

\bibitem{I} M.~Izumi, `Inclusions of simple $C^*$-algebras', J. Reine Angew. Math. \textbf{547} (2002), 97-138.

\bibitem{JOPT} J.~A.~Jeong, H.~Osaka, N.~C.~Philips and T.~Teruya, `Cancellation for inclusions of $C^*$-algebras of finite depth', Indiana Univ. Math. Journal {\bf 58} (2009), 1537-1564.

\bibitem{KK} R.~V.~Kadison and D.~Kastler, `Perturbations of von Neumann algebras. I. Stability of type', Amer. J. Mathematics {\bf 94} (1972), 38-54. 

  \bibitem{KAW} T.~Kajiwara and Y.~Watatani, `Jones index theory by Hilbert $C^*$-bimodules and $K$-theory', Trans. Amer. Math. Soc. {\bf 352} (2000), 3429-3472. 


\bibitem{KW} S.~Kawakami and Y.~Watatani, `The multiplicativity of the minimal index of Simple $C^*$-algebras', Proc. Amer. Math. Soc. {\bf 123} (1995), 2809-2813. 

\bibitem{MK} M.~Khoshkam, `Hilbert $C^*$-modules and conditional expectations on crossed products', J. Aust. Math. Soc. Ser. A {\bf 61} (1996), 106-118.

  \bibitem{KM} M.~Khoshkam and B.~Mashood, `On finiteness of the set of
  intermediate subfactors', Proc. Amer. Math. Soc. {\bf 132} (2004), 2939-2944.

  


\bibitem{Longo} R.~Longo, `Conformal subnets and intermediate subfactors', Comm. Math. Phys. {\bf 237} (2003), 7-30.

  
\bibitem{MT} B.~Mashood and K.~Taylor, `On continuity of the index of subfactors of a finite factor', J. Funct. Anal. {\bf 76} (1988), 56-66.

\bibitem{N} N.~C.~Phillips, Notes on crossed product $C^*$-algebras and minimal dynamics, 2017 .

\bibitem{Po} S.~Popa, A.~M.~Sinclair and R.~R.~Smith, `Perturbations of subalgebras of type $II_1$ factors', J. Funct. Anal. {\bf 213} (2004),
346-379.
\bibitem{TW} T.~Teruya and Y.~Watatani, `Lattice of intermediate subfactors of type III factors', Arch. Math. {\bf 68} (1997), 454-463.

\bibitem{Wat} Y.~Watatani, `Index for $C^*$-subalgebras', Memoirs of the American Mathematical Society \textbf{83} (1990).

\bibitem{Wat2} Y.~Watatani, `Lattices of intermediate subfactors', J. Funct. Anal. {\bf 140} (1996), 312-334.

\end{thebibliography}
\end{document}